\documentclass[11pt, a4paper]{article}
\usepackage{}
\usepackage{amsthm}
\usepackage{mathrsfs}
\usepackage{amsmath,amssymb,latexsym,color}
\usepackage[colorlinks,
linkcolor=blue,
anchorcolor=green,
citecolor=magenta
]{hyperref}
\usepackage{graphicx}
\usepackage{tikz}
\usetikzlibrary{calc}

\usepackage{CJK}

\usepackage{authblk}

\long\def\delete#1{}

\usepackage{color}

\definecolor{Blue}{rgb}{0,0,1}
\definecolor{Red}{rgb}{1,0,0}
\definecolor{DarkGreen}{rgb}{0,0.6,0}
\definecolor{DarkYellow}{rgb}{1,1,0.2}
\definecolor{DarkPurple}{rgb}{.6,0,1}

\usepackage{xcolor}
\usepackage[normalem]{ulem}

\usepackage{cleveref}
\crefformat{section}{\S#2#1#3}
\crefformat{subsection}{\S#2#1#3}
\crefformat{subsubsection}{\S#2#1#3}
\crefrangeformat{section}{\S\S#3#1#4 to~#5#2#6}
\crefmultiformat{section}{\S\S#2#1#3}{ and~#2#1#3}{, #2#1#3}{ and~#2#1#3}

\def\ma{\mathcal{A}}
\def\mb{\mathcal{B}}

\def\mf{\mathcal{F}}
\def\mg{\mathcal{G}}

\def\ml{\mathcal{L}}

\def\mn{\mathcal{N}}

\def\mr{\mathcal{R}}
\def\ms{\mathcal{S}}
\def\mt{\mathcal{T}}

\def\mw{\mathcal{W}}
\def\mx{\mathcal{X}}
\def\my{\mathcal{Y}}

\def\bs{\setminus}

\def\c{\choose}

\def\uuu{U^{[ck]}_{c,k-1}}
\def\uu{U^{[ck]}_{c,k}}
\def\ut{U^{[ck]}_{c,t}}
\def\tt{\theta(c,k,t+1)}
\def\ttt{\theta(c,k,t+2)}

\def\zt{\theta(c,k,z)}

\def\ge{\geqslant}
\def\le{\leqslant}

\def\ro{\romannumeral}

\numberwithin{equation}{section}

\newtheorem{thm}{Theorem}[section]
\newtheorem{lem}[thm]{Lemma}

\newtheorem{cl}{Claim}

\newtheorem{con}{Construction}

\begin{document}
	\setcounter{page}{1}
	\renewcommand{\thefootnote}{}
	\newcommand{\remark}{\vspace{2ex}\noindent{\bf Remark.\quad}}
	\renewcommand{\abovewithdelims}[2]{%
		\genfrac{[}{]}{0pt}{}{#1}{#2}}

	%-------------------  First Head  -----------------------------------------
	
	\def\qed{\hfill$\Box$\vspace{11pt}}
	
	\title {\bf  Cross-intersection theorems for uniform partitions of finite sets}

	\author[a]{Tian Yao\thanks{E-mail: \texttt{tyao@hist.edu.cn}}}
	\author[b]{Mengyu Cao\thanks{Corresponding author. E-mail: \texttt{myucao@ruc.edu.cn}}}
	\author[c]{Haixiang Zhang\thanks{E-mail: \texttt{zhang-hx22@mails.tsinghua.edu.cn}}}
	\affil[a]{School of Mathematical Sciences, Henan Institute of Science and Technology, Xinxiang 453003, China}
	\affil[b]{Institute for Mathematical Sciences, Renmin University of China, Beijing 100086, China}
	\affil[c]{Department of Mathematical Sciences, Tsinghua University, Beijing 100084, China}

	\date{}
	
	\openup 0.5\jot
	\maketitle

	\begin{abstract}
A set partition is $c$-uniform if every block has size $c$.
Two families of $c$-uniform partitions of a finite set are said to be cross $t$-intersecting if two partitions from different families share at least $t$ blocks. 
In this paper, we   establish some product-type extremal results for such cross $t$-intersecting families. Our results yield  an  Erd\H{o}s-Ko-Rado theorem and a  Hilton-Milner theorem for uniform set partitions. 
Additionally, cross $t$-intersecting families with the maximum sum of their sizes are also characterized.
		
		\vspace{2mm}
		
		\medskip
		

	\end{abstract}

\section{Introduction}

Let $n$, $k$ and $t$ be positive integers with $n\ge k\ge t$. For an $n$-set $X$, denote the set of all $k$-subsets of $X$ by ${X\choose k}$. A family $\mf\subseteq{X\choose k}$ is said to be \emph{$t$-intersecting} if $|F\cap G|\ge t$ for any $F,G\in\mf$. A $t$-intersecting family is called \emph{trivial} if all its members contain a common specified $t$-subset of $X$, and \emph{non-trivial} otherwise.
The celebrated Erd\H{o}s-Ko-Rado theorem \cite{EKR1}  states that each extremal $t$-intersecting subfamily of ${X\choose k}$ is trivial for $n>n_0(k,t)$. It is known that the smallest possible such function $n_0(k, t)$ is $(t+1)(k-t+1)$ \cite{CEKR,EKRF,Frankl--Furedi-1991,EKR2}. A type of stability result of this theorem is to determine the structure of extremal non-trivial $t$-intersecting subfamilies of ${X\choose k}$. The first result is the Hilton-Milner Theorem \cite{HM1} which describes the structure of such families for $t=1$.  Frankl \cite{HM2} determined such families for $t\ge 2$ and $n>n_1(k,t)$.  Ahlswede and Khachatrian \cite{Ahlswede-Khachatrian-1996}   provided the smallest possible $n_1(k,t)$ and gave the complete result on non-trivial intersection problems for finite sets. Recently, other large maximal non-trivial $t$-intersecting families have also been studied, see \cite{Cao-set,Han-Kohayakawa,Kostochka-Mubayi}.

A natural generalization of the $t$-intersecting family is the concept of cross $t$-intersecting families. 
Subfamilies $\mf$ and $\mg$ of ${X\choose k}$, with $|F\cap G|\ge t$ for any $F\in\mf$ and $G\in\mg$, are said to be \emph{cross $t$-intersecting}. Observe that if $\mf=\mg$, then $\mf$ is $t$-intersecting. 
The study of cross $t$-intersecting families has attracted considerable attention in extremal combinatorics, with two primary lines: one seeking to maximize the sum $|\mf|+|\mg|$, and the other aiming to maximize the product $|\mf||\mg|$. 
%They are said to be \emph{trivial} if there exists a $t$-subset of $[n]$ such that it is contained in each element of $\mf$ and $\mg$, and \emph{non-trivial} if $|\bigcap_{F\in \mf\cup\mg} F|<t$. 
%There is extensive literature on the maximum sum or product of the sizes of cross $t$-intersecting families. 
For the sum version, we refer to~\cite{F-T-1992,H-P-2025,L-Z-2025,WZ,Z-F-2024}; for the product version, see~\cite{B-2016,C-L-L-W-2024,F-W-2023,F-W-2024,H-L-W-Z-2026,T-2013,Z-W-2025}.

Intersection problems have been generalized to many mathematical objects. 
The problems on uniform partitions of finite sets are higher order extremal problems \cite{ES}.
A \emph{$c$-uniform partition} of $[ck]:=\{1,2,\dots,ck\}$ is a partition of it into $k$ blocks with equal sizes. 
Actually, a $c$-uniform partition of $[ck]$ can be viewed as a perfect matching of a complete $c$-uniform hypergraph on $ck$ vertices. The corresponding intersection problems are related to graph problems investigated by Simonovits and S\'{o}s (e.g. \cite{S-S-1980}).

We say a family of some $c$-uniform partitions of $[ck]$ is \emph{$t$-intersecting} if any two members of this family share at least $t$ blocks. 
Meagher  and Moura \cite{MM2005} proved an Erd\H{o}s-Ko-Rado theorem for $t$-intersecting families of $c$-uniform partitions of $[ck]$ for sufficiently large $c$ or $k$. 
They completely solved the case $t=1$, and conjectured a complete Erd\H{o}s-Ko-Rado theorem based on the Ahlswede-Khachatrian theorem \cite{CEKR}. 
Some algebraic proofs for the case $c=2$ and $t=1$ were given in \cite{GM,Lindzey}, and 
a more precise result for $c=t=2$ was obtained by Fallat et al. \cite{Fallat}. 
We remark here that a $t$-intersecting family of permutations can be seen as a special family of $2$-uniform partitions, and Ellis et al. \cite{JAMS} showed the corresponding Erd\H{o}s-Ko-Rado theorem.  For the results on the non-uniform case, we refer readers to \cite{ES,KR,KWEJC,KWJCTA}.

Two families of $c$-uniform partitions of $[ck]$ are said to be \emph{cross $t$-intersecting} if two partitions from different families have at least $t$ blocks in common. 
In this paper, we first show two product-type extremal results for cross $t$-intersecting families. 
One of our results is an Erd\H{o}s-Ko-Rado type theorem. 
For convenience, let $U^{[ck]}_{c,\ell}$ denote the set of all families consisting of $\ell$ pairwise disjoint $c$-subsets of $[ck]$.

 \begin{thm}\label{EKR-perfect}
 	Let $c$, $k$ and $t$ be positive integers with $c\ge3$ and $k\ge t+2$. Suppose $\mf$ and $\mg$ are cross $t$-intersecting subfamilies of $\uu$ such that $|\mf||\mg|$ is maximum.
 	If $c\ge3+2\log_2t$ or $k\ge2t+2$, then $\mf=\mg=\{F\in\uu:T\subseteq F\}$ for some  $T\in U^{[ck]}_{c,t}$.
 \end{thm}

Theorem \ref{EKR-perfect} can be viewed as a product version of the Erd\H{o}s-Ko-Rado theorem, which generalizes  
the work of Meagher and Moura \cite{MM2005}. 
%They require that $c$ or $k$ is sufficiently large. 
Before stating our second result, we introduce three cross $t$-intersecting families.

\begin{con}\label{n1}
	Suppose $c$, $k$ and $t$ are positive integers with $c\ge3$ and $k\ge t+3$.
	Let $T\in\ut$ and $L,M\in\uuu$ with $T\subseteq L\cap M$ and $|L\cap M|\ge t+\min\{t,2\}$. Write
	$$\mn_1(T,L,M)=\left\{F\in\uu: T\subseteq F,\ |F\cap L|\ge t+1\right\}\cup\left\{F\in\uu: T\not\subseteq F,\ |F\cap M|=k-2\right\}.$$
	Notice that $|\bigcap_{F\in\mn_1(T,L,M)}F|<t$, and $\mn_1(T,L,M)$ and $\mn_1(T,M,L)$ are cross $t$-intersecting.
\end{con}

\begin{con}\label{n2}
	Suppose $c$, $k$ and $t$ are positive integers with $c\ge3$ and $k\ge t+3$.
	Let $Z\in U^{[ck]}_{c,t+2}$. Write
	$$\mn_2(Z)=\left\{F\in\uu: |F\cap Z|\ge t+1\right\}.$$
		Notice that $|\bigcap_{F\in\mn_2(Z)}F|<t$, and $\mf$ and $\mg$ are cross $t$-intersecting if $\mf=\mg=\mn_2(Z)$.
\end{con}

\begin{con}\label{n3}
	Suppose $c$, $k$ and $t$ are positive integers with $c\ge3$ and $k\ge 4$.
	Let $A_1=\{e_1,e_2\}$, $A_2=\{e_3,e_4\}$, $B_1=\{e_1,e_3\}$, $B_2=\{e_2,e_4\}$, $C=\{e_1,e_4\}\in U_{c,2}^{[ck]}$. Write
	$$\mn_3(A_1,A_2,C)=\left\{F\in\uu: A_1\subseteq F,\ \mbox{or}\ A_2\subseteq F,\ \mbox{or}\ C\subseteq F\right\}.$$
		Notice that $\bigcap_{F\in\mn_3(A_1,A_2,C)}F$ is empty, and $\mn_3(A_1,A_2,C)$ and $\mn_3(B_1,B_2,C)$ are cross $1$-intersecting.
\end{con}

Our second result is stated as follows, from which we can derive an analogue of the Hilton-Milner theorem for uniform set partitions. %It can be seen as a  product version of the Hilton-Milner theorem.

\begin{thm}\label{HM-perfect}
	Let $c$, $k$ and $t$ be positive integers with $c\ge 6$ and $k\ge t+3$.
	Suppose $\mf$ and $\mg$ are cross $t$-intersecting subfamilies of $\uu$ such that both $|\bigcap_{F\in\mf}F|$ and $|\bigcap_{G\in\mg}G|$ are less than $t$.
	If $c\ge4\log_2t+7$ or $k\ge2t+3$, and $|\mf||\mg|$ takes the maximum value, then  one of the following holds.
	\begin{itemize}
		\item[\rm{(\ro1)}]  $\mf=\mn_1(T,L,M)$ and $\mg=\mn_1(T,M,L)$ for some $T\in\ut$ and $L,M\in\uuu$ with $T\subseteq L\cap M$ and $|L\cap M|\ge t+\min\{t,2\}$.
		\item[\rm{(\ro2)}] $\mf=\mg=\mn_2(Z)$ for some  $Z\in U^{[ck]}_{c,t+2}$.
		\item[\rm{(\ro3)}] $\mf=\mn_3(A_1,A_2,C)$ and $\mg=\mn_3(B_1,B_2,C)$ for some $A_1=\{e_1,e_2\},A_2=\{e_3,e_4\},B_1=\{e_1,e_3\},B_2=\{e_2,e_4\},C=\{e_1,e_4\}\in U_{c,2}^{[ck]}$.
	\end{itemize}
	Moreover, if $(k,t)\in\{(4,1),(5,1)\}$, then {\rm(\ro1)}, {\rm(\ro2)} or {\rm(\ro3)} holds; if $k=t+3$ and $t\ge 2$, then {\rm(\ro1)} or {\rm(\ro2)} holds; if $t+4\le k\le2t+3$ with $(k,t)\neq(5,1)$, then {\rm(\ro2)} holds; if $k\ge2t+4$, then {\rm(\ro1)} holds.
\end{thm}

In 2013,  Wang and Zhang \cite{WZ} completely solved the problem of maximizing the sum of sizes of cross $t$-intersecting families for sets, vector spaces and symmetric groups. 
These problems were reduced to describing all fragments in bipartite graphs. 
Inspired by their work, we investigate the fragments in a specified bipartite graph, and then characterize cross $t$-intersecting families of uniform set partitions with maximum sum of their sizes.

\begin{thm}\label{SUM}
	Let $c$, $k$ and $t$ be positive integers with $c\ge3$ and $k\ge t+2$. Suppose $\mf$ and $\mg$ are cross $t$-intersecting subfamilies of $\uu$ such that $|\mf|\le|\mg|$ and $|\mf|+|\mg|$ is maximum.
	If $c\ge3+2\log_2t$ or $k\ge2t+2$, then  $\mf=\{C\}$ and $\mg=\{F\in\uu:|F\cap C|\ge t\}$ for some $C\in\uu$.
\end{thm}

The rest of this paper is organized as follows. In Sections \ref{T11}--\ref{T13}, we prove Theorems \ref{EKR-perfect}--\ref{SUM}, respectively. All maximal cross $t$-intersecting subfamilies of $\uu$ for $k=t+2$ are characterized in Section \ref{t+2}, which shows that the hypothesis $k\ge t+3$ in Theorem \ref{HM-perfect} is essential. Finally in Section \ref{COMPUTION}, some inequalities used in our proofs are presented.

\section{Proof of Theorem \ref{EKR-perfect}}\label{T11}

Let $\mf\subseteq U^{[ck]}_{c,\ell}$ and $S\in U^{[ck]}_{c,s}$. 
We say $S$ is a \emph{$t$-cover} of $\mf$ if $|S\cap F|\ge t$ for each $F\in\mf$, and the minimum size $\tau_t(\mf)$ of a $t$-cover of $\mf$ is the \emph{$t$-covering number} of $\mf$.
For convenience, write
$\mf_S=\{F\in\mf: S\subseteq F\}$
and denote the set of all minimum $t$-covers of $\mf$ by $\mt_t(\mf)$.

\begin{lem}\label{fangsuo}
	Let $c$, $k$ and $t$ be positive integers with $c\ge2$ and $k\ge t+2$. Suppose $\mf\subseteq\uu$ and $G\in\uu$ is a $t$-cover of $\mf$. If  $|G\cap S|=r<t$ for some $S\in U^{[ck]}_{c,s}$, then
	for each $i\in\{1,2,\dots,t-r\}$, there exists $R\in U^{[ck]}_{c,s+i}$ such that $S\subseteq R$ and
	$$|\mf_S|\le{k-s\c i}|\mf_R|.$$
\end{lem}
\begin{proof}
	If $\mf_S=\emptyset$, then there is nothing to prove. Thus we may suppose that $\mf_S\neq\emptyset$.
	For $i\in\{1,2,\dots,t-r\}$, set
	$$\mr_i=\{R\in U^{[ck]}_{c,s+i}: S\subseteq R\subseteq G\cup S\}.$$	
	Let $F\in\mf_S$. Since $|F\cap G|\ge t$, we know $F$ contains at least $t-r$ blocks in $G\bs S$.
	We further conclude that
	 $$\mr_i\neq\emptyset,\quad\mf_S=\{F\in\mf_S:|F\cap(G\cup S)|\ge s+t-r\}.$$
Then
	$$|\mf_S|=\left|\bigcup_{R\in\mr_i}\mf_R\right|\le\sum_{R\in\mr_i}|\mf_R|.$$	
	Since $G\in\uu$ and $|G\cap S|<t$, there are at most $k-s$ blocks in $G\bs S$ which are disjoint with each block in $S$,
	implying that $|\mr_i|\le{k-s\c i}$. Pick $R\in\mr_i$ such that $|\mf_R|$ takes the maximum value. Then
	$$|\mf_S|\le{k-s\c i}|\mf_R|,$$
	as desired.
\end{proof}

Let $c$, $k$, $t$ and $z$  be positive integers with $c\ge2$, $k\ge t+2$ and $k\ge z$. Denote the number of the elements in $\uu$ containing a fixed member of $U^{[ck]}_{c,z}$ by $\zt$. Observe that
$$\zt=\dfrac{1}{(k-z)!}\prod_{i=z}^{k-1}{(k-i)c\c c}.$$
We also write
$$g(c,k,t,z)=\zt\binom{z}{t}\prod_{j=1}^{z-t}\left(k-(t+j-1)\right).$$

\begin{lem}\label{fangsuo1}
Let $c$, $k$ and $t$ be positive integers with $c\ge2$ and $k\ge t+2$.	 If $\mf$ and $\mg$ are cross $t$-intersecting subfamilies of $\uu$, then	
\begin{equation}\label{eq-fangsuo1-1}
|\mf|\le\theta(c,k,\tau_t(\mg)){\tau_t(\mf)\c t}\prod_{j=1}^{\tau_t(\mg)-t}\left(k-(t+j-1)\right).
\end{equation}

\end{lem}
\begin{proof}
	Suppose that $\tau_t(\mf)=z$ and $Z\in\mt_t(\mf)$. Observe that
	\begin{equation}\label{F-FZ}
		\mf\subseteq\bigcup_{Y\in{Z\c t}}\mf_Y.
	\end{equation}
	  If $\tau_t(\mg)=t$, then \eqref{eq-fangsuo1-1} follows from \eqref{F-FZ}. Next assume that $\tau_t(\mg)\ge t+1$.
	
	Pick $Y_0\in{Z\c t}$ with $\mf_{Y_0}\neq\emptyset$. 
	Notice that $Y_0$ is not a $t$-cover of $\mg$. Then there exists $G_0\in\mg$ with $|G_0\cap Y_0|<t$.
	Since $G_0$ is a $t$-cover of $\mf$, by Lemma \ref{fangsuo}, we have
	$$|\mf_{Y_0}|\le\left(k-t\right)|\mf_{Y_1}|$$
	for some $Y_1\in U^{[ck]}_{c,t+1}$. Using Lemma \ref{fangsuo} repeatedly, we get  $Y_0\in U^{[ck]}_{c,t}$, $Y_1\in U^{[ck]}_{c,t+1}$, \dots, $Y_{\tau_t(\mg)-t}\in U^{[ck]}_{c,\tau_t(\mg)}$ with
	$$|\mf_{Y_{j-1}}|\le\left(k-(t+j-1)\right)|\mf_{Y_{j}}|$$
	for each $j\in\{1,\dots,\tau_t(\mg)-t\}$. Then
	\begin{align*}
		|\mf_{Y_0}|
		\le\prod_{j=1}^{\tau_t(\mg)-t}\left(k-(t+j-1)\right)\cdot|\mf_{Y_{\tau_t(\mg)-t}}|
		\le\theta(c,k,\tau_t(\mg))\prod_{j=1}^{\tau_t(\mg)-t}\left(k-(t+j-1)\right).
	\end{align*}
	This together with \eqref{F-FZ} yields \eqref{eq-fangsuo1-1}.
\end{proof}

In the rest of this papaer, for $\ell,m\in\{t,t+1,\dots,k\}$, we also say $\ma\subseteq U^{[ck]}_{c,\ell}$ and $\mb\subseteq U^{[ck]}_{c,m}$ are cross $t$-intersecting if $|A\cap B|\ge t$ for any $A\in\ma$ and $B\in\mb$.

\begin{lem}\label{cover-intersection}
Let $c$, $k$ and $t$ be positive integers with $c\ge2$ and $k\ge t+2$.	Suppose $\mf$ and $\mg$ are maximal cross $t$-intersecting subfamilies of $\uu$ with $\max\{\tau_t(\mf),\tau_t(\mg)\}\le k-2$. Then $\mt_t(\mf)$ and $\mt_t(\mg)$ are cross $t$-intersecting.	
\end{lem}
\begin{proof}
Pick $T_1\in\mt_t(\mf)$ and $T_2\in\mt_t(\mg)$. Write
$$\alpha=\max\left\{r\in\mathbb{N}: T_2\cup A\in  U^{[ck]}_{c,t+r}\ \mbox{for some}\   A\in{T_1\bs T_2\choose r}\right\}.$$
Then there exist $e_1,\dots,e_{k-\tau_t(\mg)}\in{[ck]\c c}$ with $e_1,\dots,e_\alpha\in T_1$ and
$T_2\cup\{e_1,\dots,e_{k-\tau_t(\mg)}\}\in\uu$.
For each $i\in\{1,2,\dots,k-\tau_t(\mg)\}$, write $$e_i=\{v_i,v_{i+(k-\tau_t(\mg))},\dots,v_{i+(c-1)(k-\tau_t(\mg))}\},\quad f_i=\{v_{(i-1)c+1},v_{(i-1)c+2},\dots,v_{ic}\}.$$
Let $F=T_2\cup\{f_1,\dots,f_{k-\tau_t(\mg)}\}$.
We have $F\in\uu$ and $F\cap\{e_1,e_2,\dots,e_{k-\tau_t(\mg)}\}=\emptyset$ from $k\ge\tau_t(\mg)+2$, implying that $F\cap T_1=T_2\cap T_1$. Note that $F\in\mf$ by the maximality of $\mf$ and $\mg$.
Then it follows from  $T_1\in\mt_t(\mf)$ that $|T_2\cap T_1|=|F\cap T_1|\ge t$.
\end{proof}

\begin{proof}[{\bf Proof of Theorem \ref{EKR-perfect}}]
	Suppose that $\mf$ and $\mg$  are $t$-intersecting subfamilies of $\uu$ and $|\mf||\mg|$ takes the maximum value. Observe that
	\begin{equation}\label{PF-EKR}
		|\mf||\mg|\ge (g(c,k,t,t))^2.
	\end{equation}
	It is sufficient to show $\tau_t(\mf)=\tau_t(\mg)=t$, which together with Lemma \ref{cover-intersection} yields Theorem \ref{EKR-perfect}.
	
%	Suppose  $k=t+2$.  By  Theorem \ref{thm-t+2}, if $\max\{\tau_t(\mf),\tau_t(\mg)\}\ge t+1$, then $$|\mf||\mg|\le{t+2\c2}\left(\dfrac{1}{2}{2c\c c}-1\right)+1<\dfrac{1}{2}{2c\c c}{t+2\c2}<4t^2\left(\dfrac{1}{2}{2c\c c}\right)<\left(\dfrac{1}{2}{2c\c c}\right)^2.$$
%	This contradicts \eqref{PF-EKR}. Thus $\tau_t(\mf)=\tau_t(\mg)=t$.
	
	%Suppose $k\ge t+3$.
	By Lemma \ref{fangsuo1}, we have $|\mf||\mg|\le g(c,k,t,\tau_t(\mf))g(c,k,t,\tau_t(\mg))$.
		If $\max\{\tau_t(\mf),\tau_t(\mg)\}\ge t+1$, then by Lemma \ref{key-bidaxiao}, we have $|\mf||\mg|<(g(c,k,t,t))^2$, a contradiction to \eqref{PF-EKR}. So $\tau_t(\mf)=\tau_t(\mg)=t$, 
		as desired.
\end{proof}

\section{Proof of Theorem \ref{HM-perfect}}\label{T12}

In this section, we investigate maximal cross $t$-intersecting subfamilies $\mf$ and $\mg$ of $\uu$ with $|\mf||\mg|$ being maximum under the condition that $\min\{\tau_t(\mf),\tau_t(\mg)\}\ge t+1$.

Write
\begin{equation}\label{f0}
	f_0(c,k,t)=(k-t-1)\tt-{k-t-1\c2}\ttt.
\end{equation}
We remark here that  the sizes of families stated in Constructions \ref{n1} and \ref{n2} are only related to $c$, $k$ and $t$. Consequently, in the rest of this paper, let $f_1(c,k,t)$ and $f_2(c,k,t)$ denote that sizes of families stated in Constructions \ref{n1} and \ref{n2}, respectively.
These two constructions implies that $|\mf||\mg|\ge\max\{(f_1(c,k,t))^2,(f_2(c,k,t))^2\}$.
The following lemma shows that both $\tau_t(\mf)$ and $\tau_t(\mg)$ are equal to $t+1$.

\begin{lem}\label{>t+2}
	Let $c$, $k$ and $t$ be positive integers with $c\ge5$ and $k\ge t+3$. Suppose $\mf$ and $\mg$ are cross $t$-intersecting subfamilies of $\uu$ with $\tau_t(\mg)\ge\tau_t(\mf)\ge t+1$ and $|\mf||\mg|\ge(f_1(c,k,t))^2$.
	If $c\ge4\log_2t+7$ or $k\ge 2t+3$, then $\tau_t(\mf)=\tau_t(\mg)=t+1$.
\end{lem}
\begin{proof}
	Suppose for contradiction that $\tau_t(\mg)\ge t+2$. By Lemma \ref{fangsuo1}, we have
	\begin{equation}\label{>t+2-1}
		|\mf||\mg|\le g(c,k,t,\tau_t(\mf))g(c,k,t,\tau_t(\mg)).
	\end{equation}

	Assume $(k,\tau_t(\mg))\neq(t+3,t+3)$. We have $k\ge t+4$ or $(k,\tau_t(\mg))=(t+3,t+2)$, and obtain
	$$|\mf||\mg|\le g(c,k,t,t+1)g(c,k,t,t+2)$$
	from Lemma \ref{key-bidaxiao} and \eqref{>t+2-1}.
	This together with Lemma \ref{=t+1=t+1} (\ro1) produces $|\mf||\mg|<(f_0(c,k,t))^2$.

	Assume $k=\tau_t(\mg)=t+3$. By Lemma \ref{key-bidaxiao} and \eqref{>t+2-1}, we have
	$$|\mf||\mg|\le g(c,k,t,t+1)g(c,k,t,t+3).$$
	It follows from Lemma \ref{=t+1=t+1} (\ro2)  that $|\mf||\mg|<(f_0(c,t+3,t))^2$.

	In summary, $|\mf||\mg|<(f_0(c,k,t))^2$. Then by Lemma \ref{f0f1} (\ro1), we get $|\mf||\mg|<(f_1(c,k,t))^2$.
	This contradicts the assumption that $|\mf||\mg|\ge (f_1(c,k,t))^2$. So $\tau_t(\mf)=\tau_t(\mg)=t+1$.
\end{proof}

By Lemmas \ref{cover-intersection} and \ref{>t+2}, we have $\tau_t(\mt_t(\mf)),\tau_t(\mt_t(\mg))\in\{t,t+1\}$.

\subsection{The case $(\tau_t(\mt_t(\mf)),\tau_t(\mt_t(\mg)))\neq(t+1,t+1)$}

\begin{lem}\label{cover}
	Let $c$, $k$ and $t$ be positive integers with $c\ge2$ and $k\ge t+3$. Suppose that $\mf\subseteq\uu$ with $\tau_t(\mf)=t+1$, and $A\in U^{[ck]}_{c,t}$ with $(\mt_t(\mf))_A\neq\emptyset$. Set $E=\bigcup_{S\in (\mt_t(\mf))_A}S$. Then there exists $m\in\{t+1,t+2,\dots,k-1\}$ such that the following hold.
	\begin{itemize}
		\item[\rm{(\ro1)}] $E\in U^{[ck]}_{c,m}$ and $(\mt_t(\mf))_A=\{S\in U^{[ck]}_{c,t+1}: A\subseteq S\subseteq E\}$.
		\item[\rm{(\ro2)}] $|F\cap E|=m-1$ for each $F\in\mf\bs\mf_A$.
	\end{itemize}
\end{lem}
\begin{proof}
	If $|(\mt_t(\mf))_A|=1$, then it is obvious that $m=t+1$, and (\ro1) and (\ro2) hold. Next assume that $|(\mt_t(\mf))_A|\ge2$.
	
	Let $S_1$ and $S_2$ be distinct members of $(\mt_t(\mf))_A$.
	Since $\tau_t(\mf)=t+1$, there exists $F\in\mf$ with $A=S_1\cap S_2\not\subseteq F$.
	This together with $|S_1\cap F|\ge t$ and $|S_2\cap F|\ge t$ yields $|A\cap F|=t-1$, $S_1\Delta S_2\subseteq F$ and $S_1\Delta S_2\in U^{[ck]}_{c,2}$.
	Hence $S_1\cup S_2\in U^{[ck]}_{c,t+2}$.
	We further conclude that $E\in U^{[ck]}_{c,m}$ for some $m\ge t+2$. 
	
	Since $|A\cap F|=t-1$ and $S\bs A\subseteq F$ for each $S\in(\mt_t(\mf))_A$, 
	we have $|F\cap E|=|A\cap F|+(|E|-t)=m-1$. Then (\ro2) holds.

	Observe that $m\le k$. If $m=k$, then $|F\cap E|=k-1$, which implies that $F=E$, a contradiction to the assumption that $T\not\subseteq F$. Thus $m\le k-1$. Now we conclude that $m\in\{t+2,t+3,\dots,k-1\}$ and (\ro1) follows.
	\end{proof}

\begin{lem}\label{B}
	Let $c$, $k$ and $t$ be positive integers with $c\ge2$ and $k\ge t+3$.
	Suppose $\mf$ and $\mg$ are cross $t$-intersecting subfamilies of $\uu$ with $\tau_t(\mf)=\tau_t(\mg)=t+1$.
	If $\mb=\{F\in\mf: P\not\subseteq F,\ \forall P\in\mt_t(\mg)\}$, then
	$$\dfrac{|\mb|}{\tt}\le\dfrac{3(t+1)(k-t-1)^3}{2{(k-t-1)c\c c}}.$$
\end{lem}
\begin{proof}
	If $\mb=\emptyset$, then there is nothing to prove. Next suppose that $\mb\neq\emptyset$.
	Let $S\in\mt_t(\mf)$. We have
	\begin{equation}\label{B-1}
		\mb\subseteq\bigcup_{T\in{S\c t}}\mb_T.
			\end{equation}
	Choose $T\in{S\c t}$ with $\mb_T\neq\emptyset$.
	
	It follows from $\tau_t(\mg)=t+1$ that $|T\cap G|<t$ for some $G\in\mg$. 
	Since $|G\cap F|\ge t$ for each $F\in\mb_T$, we have
	\begin{equation}\label{B-2}
		\mb_T\subseteq\bigcup_{H=T\cup\{e\}\in U^{[ck]}_{c,t+1},\ e\in G\bs T}\mb_H.
			\end{equation}
    Observe that 	the number of blocks in $G\bs T$ which are disjoint with each block in $T$ is at most $k-t$.

    Let $H=T\cup\{e\}\in U^{[ck]}_{c,t+1}$ for some $e\in G\bs T$ with $\mb_H\neq\emptyset$.
    The definition of $\mb$ implies $H\not\in\mt_t(\mg)$. 
    By Lemma \ref{fangsuo}, there exists $R\in U^{[ck]}_{c,t+2}$ such that
    $$|\mb_H|\le(k-t-1)|\mb_R|\le(k-t-1)\ttt.$$
    This together with \eqref{B-1}, \eqref{B-2} and $2(k-t)\le3(k-t-1)$ produces
    	\begin{equation*}
    	\begin{aligned}
    		\dfrac{|\mb|}{\tt}\le\dfrac{(t+1)(k-t)(k-t-1)^2}{{(k-t-1)c\c c}}\le\dfrac{3(t+1)(k-t-1)^3}{2{(k-t-1)c\c c}},
    			\end{aligned}
    	\end{equation*}
    	as desired.
\end{proof}

%It is routine to check that
%$$f_2(c,k,t)=(t+2)\tt-(t+1)\ttt.$$

\begin{lem}\label{tt}
	Let $c$, $k$ and $t$ be positive integers with $c\ge6$ and $k\ge t+3$.
	Suppose $\mf$ and $\mg$ are maximal cross $t$-intersecting subfamilies of $\uu$ with $\tau_t(\mf)=\tau_t(\mg)=t+1$,  $(\tau_t(\mt_t(\mf)),\tau_t(\mt_t(\mg)))\neq(t+1,t+1)$ and $|\mf||\mg|\ge\max\{(f_1(c,k,t))^2,(f_2(c,k,t))^2\}$.
	If $c\ge4\log_2t+7$ or $k\ge2t+3$,  then $\tau_t(\mt_t(\mf))=\tau_t(\mt_t(\mg))=t$.
\end{lem}
\begin{proof}
	Recall that $\tau_t(\mt_t(\mf)),\tau_t(\mt_t(\mg))\in\{t,t+1\}$. 
	 W.l.o.g., suppose for contradiction that $\tau_t(\mt_t(\mf))=t$ and $\tau_t(\mt_t(\mg))=t+1$.
	
	 \medskip
	 \noindent{{\bf Case 1.} $|\mt_t(\mf)|=1$}
	 \medskip
	
	 In this case, it follows from Lemmas \ref{cover-intersection} and \ref{cover} that $|\mt_t(\mg)|\le(t+1)(k-t-1)$.
	 Therefore, by Lemma \ref{B}, we obtain
	\begin{equation*}
		\begin{aligned}
			\dfrac{|\mf|}{\tt}
		<(t+1)(k-t-1)+\dfrac{3(t+1)}{2(k-t-1)^{c-3}},\quad
			\dfrac{|\mg|}{\tt}<1+\dfrac{3(t+1)}{2(k-t-1)^{c-3}}.	
		\end{aligned}
		\end{equation*}
	These together with Lemmas  \ref{=t=t} (\ro1) and \ref{f0f1} (\ro1) produce $$|\mf||\mg|<\max\{(f_0(c,k,t))^2,(f_2(c,k,t))^2\}\le\max\{(f_1(c,k,t))^2,(f_2(c,k,t))^2\},$$ a contradiction.
	
	 \medskip
	\noindent{{\bf Case 2.} $|\mt_t(\mf)|\ge2$}
	\medskip
	
	Since $\tau_t(\mt_t(\mf))=t$, from Lemma \ref{cover-intersection}, there exist distinct $S_1,S_2\in\mt_t(\mf)$ with $|S_1\cap S_2|=t$.
	 By Lemma \ref{cover}, we have $S_1\cup S_2\in U^{[ck]}_{c,t+2}$. This together with $\tau_t(\mt_t(\mg))=t+1$ and Lemma \ref{cover-intersection} yields
	$$\mt_t(\mg)\bs(\mt_t(\mg))_{S_1\cap S_2},\mt_t(\mf)\subseteq{S_1\cup S_2\c t+1}.$$
	Notice that $|(\mt_t(\mg))_{S_1\cap S_2}|\le k-t-1$ from Lemma \ref{cover}.
 	Then
	$$|\mt_t(\mf)|\le2,\quad|\mt_t(\mg)|\le(k-t-1)+t=k-1,$$
 	and it follows from Lemma \ref{B} that
	\begin{equation*}
		\begin{aligned}
			\dfrac{|\mf|}{\tt}<(k-1)+\dfrac{3(t+1)}{2(k-t-1)^{c-3}},\quad
			\dfrac{|\mg|}{\tt}<2+\dfrac{3(t+1)}{2(k-t-1)^{c-3}}.	
		\end{aligned}
	\end{equation*}
 By Lemma \ref{=t=t} (\ro2) and Lemma \ref{f0f1} (\ro1), 	we further conclude that $$|\mf||\mg|<\max\{(f_0(c,k,t))^2,(f_2(c,k,t))^2\}\le\max\{(f_1(c,k,t))^2,(f_2(c,k,t))^2\},$$ a contradiction.
\end{proof}

\begin{lem}\label{ttt}
	Let $c$, $k$ and $t$ be positive integers with $c\ge6$ and $k\ge t+3$.
	Suppose $\mf$ and $\mg$ are maximal cross $t$-intersecting subfamilies of $\uu$ with $\tau_t(\mf)=\tau_t(\mg)=t+1$,   $\tau_t(\mt_t(\mf))=\tau_t(\mt_t(\mg))=t$ and $|\mf||\mg|\ge(f_1(c,k,t))^2$.
	If $c\ge4\log_2t+7$ or $k\ge2t+3$,  then $\mf=\mn_1(T,L,M)$ and $\mg=\mn_1(T,M,L)$
		for some $T\in U^{[ck]}_{c,t}$ and $L,M\in\uuu$ with $T\subseteq L\cap M$ and $L\cap M\ge t+\min\{t,2\}$. In particular, if $k=t+3$, then $L\neq M$.
\end{lem}
\begin{proof}
	Let $L$ and $M$ denote the unions of all members of $\mt_t(\mf)$ and $\mt_t(\mg)$, respectively. By Lemma \ref{cover}, we have $|L|,|M|\in\{t+1,t+2,\dots,k-1\}$.
	
	Suppose $(|L|,|M|)\neq(k-1,k-1)$. By Lemma \ref{B}, we have
	$$\dfrac{|\mf|}{\tt}<|M|-t+\dfrac{3(t+1)}{2(k-t-1)^{c-3}},\quad\dfrac{|\mg|}{\tt}<|L|-t+\dfrac{3(t+1)}{2(k-t-1)^{c-3}}.$$
	Then $(f_0(c,k,t))^2-|\mf||\mg|>0$ follows Lemma \ref{=t=t} (\ro3). 
	This together with Lemma \ref{f0f1} (\ro1) yields a contradiction to the assumption $|\mf||\mg|\ge(f_1(c,k,t))^2$. Hence $|L|=|M|=k-1$.
		
		Let $T,T'\in U^{[ck]}_{c,t}$ be $t$-covers of $\mt_t(\mf),\mt_t(\mg)$, respectively.
		Notice that $\mt_t(\mf)$ and $\mt_t(\mg)$ are cross $t$-intersecting by Lemma \ref{cover-intersection}.
		For $S\in\mt_t(\mf)$, we have $|S\cap T'|\ge t-1$, and $\mt_t(\mg)\subseteq{S\cup T'\choose t+1}$ if $|S\cap T'|= t-1$.

		\medskip
		\noindent{\bf Case 1.} $|S\cap T'|=t-1$ for some $S\in\mt_t(\mf)$
		\medskip
		
		By $k\ge t+3$, we have $|S\cup T'|=t+2\le k-1\le|S\cup T'|$, which implies $k=t+3$ and $M=S\cup T'\in U^{[ck]}_{c,t+2}$.
		Observe that $T\neq T'$. It follows from $|\mt_t(\mg)|=2$ and Lemma \ref{cover-intersection} that $|T\cap S'|=t-1$ for some  $S'\in\mt_t(\mg)$.
		Then $L=T\cup S'\subseteq S\cup T'=M$. We further obtain $L=M$. We claim that
		\begin{equation}\label{clt+1}
		\mf,\mg\subseteq\left\{F\in\uu: |F\cap L|\ge t+1\right\}.
			\end{equation}
	Since $t-2\le|T\cap T'|\le t-1$, w.l.o.g., let
	$$L=\{e_1,e_2,\dots,e_{t+2}\},\ T=\{e_1,\dots,e_t\},\ e_1\not\in T',\ \{e_3,\dots,e_{t+1}\}\subseteq T',\   e_{t+3}=[ck]\bs\bigcup_{i=1}^{t+2}e_i.$$
	Suppose for contradiction that $|F\cap L|=t$ for some $F\in\mf$. Then each member of $\mt_t(\mf)$ contains $F\cap L$. This together with $|L|=t+2$ yields $F\cap L=T$.

	Assume that there exists $G_1\in\mg_{T'}$ with $|G_1\cap L|\ge t+1$. If $T\not\subseteq G_1$, then we may assume that $e_1\not\in G$. Thus $G_1\cap L=\{e_2,\dots,e_{t+2}\}$, and each block in $G_1\bs L$ intersects $e_1$ and $e_{t+3}$ at the same time. 
	 Notice that  $e_{t+1},e_{t+2}\not\in F$. 
	 We further conclude $|F\cap G_1|\le t-1$, a contradiction. Thus $T\subseteq G_1$.

	Pick $G_2\in\mg\bs\mg_{T'}$. We have $|G_2\cap L|=t+1$ by Lemma \ref{cover}, and $e_1\in G_2$. If $T\not\subseteq G_2$, then $e_i\not\in G_2$ for some $i\in\{2,\dots,t\}$, implying that $G_2\cap L=L\bs\{e_i\}$ and $(G_2\bs L)\cap F=\emptyset$. This together with $e_{t+1},e_{t+2}\not\in F$ yields $|F\cap G_2|=t-1$, a contradiction. 
	Hence $T\subseteq G_2$.
	
	Since $\tau_t(\mg)=t+1$, there exists $G_3\in\mg_{T'}$ with $G_3\cap L=T'$. 
	Notice that $|T'\cap\{e_2,e_{t+2}\}|=1$. 
	If $e_{t+2}\in T'$, then $F\cap G_3\subseteq\{e_3,\dots,e_t,e_{t+3}\}$, a contradiction to the fact that $|F\cap G_3|\ge t$.
	Hence $e_2\in T'$. 
	By $|F\cap G_3|\ge t$ and $e_{t+1}\not\in F$, we conclude that $F$ contains exactly one element $e_0$ in ${e_{t+2}\cup e_{t+3}\c c}\bs\{e_{t+2}\}$ and $e_0\in F\cap G_3$. 
	Notice that $|T\cap T'|=t-1$. 
	Furthermore, $T\cup\{e_0\}$ is a $t$-cover of $\mg_{T'}$.

	In summary, $T\cup\{e\}\in\mt_t(\mg)$ for some $e\in{e_{t+2}\cup e_{t+3}\c c}\bs\{e_{t+2}\}$. This contradicts the fact that $M=\{e_1,e_2,\dots,e_{t+2}\}$. So \eqref{clt+1} holds. Then by the maximality of $\mf$ and $\mg$, we have
	$$\mf=\mg=\left\{F\in\uu: |F\cap L|\ge t+1\right\},\quad\mt_t(\mf)=\mt_t(\mg)={L\c t+1},$$
	a contradiction to the assumption that $\tau_t(\mt_t(\mf))=\tau_t(\mt_t(\mg))=t$.
	
		\medskip
	\noindent{\bf Case 2.} $|S\cap T'|=t$ for each $S\in\mt_t(\mf)$
	\medskip
	
	In this case, we have $T'=T$ by $|\mt_t(\mf)|=k-t-1\ge2$. Let $F\in\mf_T$ and $G\in\mg\bs\mg_T$.
	Note that $|G\cap M|=k-2$ by Lemma \ref{cover}, and each block in $G\bs M$ is not a block in $F$. Thus
	$$\mf_T\subseteq\{F\in\uu: T\subseteq F,\ |F\cap M|\ge t+1\},\quad\mf\subseteq\mn_1(T,L,M)$$
	Similarly, we have
	$$\mg_T\subseteq\{G\in\uu: T\subseteq G,\ |F\cap L|\ge t+1\},\quad\mg\subseteq\mn_1(T,M,L).$$
	Since $|\mf||\mg|\ge(f_1(c,k,t))^2$, we have $\mf=\mn_1(T,L,M)$ and $\mg=\mn_1(T,M,L)$. Then the families
	$\{F\in\uu: T\not\subseteq F,\ |F\cap L|=k-2\}$ and $\{G\in\uu: T\not\subseteq G,\ |G\cap M|=k-2\}$ are cross $t$-intersecting. We further conclude that $|L\cap M|\ge t+\min\{t,2\}$.	
	Observe that $\mn_1(T,M,M)=\mn_2(M)$ when $k=t+3$, and $\tau_t(\mt_t(\mn_2(M)))=t+1$. Then $k=t+3$ implies that $L\neq M$.
\end{proof}

\subsection{The case $(\tau_t(\mt_t(\mf)),\tau_t(\mt_t(\mg)))=(t+1,t+1)$}

%For $\ell<k$, the definition of subfamilies of $U_{c,\ell}^{[ck]}$ with (cross) $t$-intersection property is  similar to those on $\uu$. 
To deal with the situation $(\tau_t(\mt_t(\mf)),\tau_t(\mt_t(\mg)))=(t+1,t+1)$, we need the following two properties of cross $t$-intersecting subfamilies of $U_{c,t+1}^{[ck]}$.

\begin{lem}\label{RS-intersecting}
	Let $c$, $k$ and $t$ be positive integers with $c\ge2$ and $k\ge t+3$.
	Suppose $\mr$ and $\ms$ are cross $t$-intersecting subfamilies of $U^{[ck]}_{c,t+1}$ with $\tau_t(\mr)=\tau_t(\ms)=t+1$. Then $\mr$ is $t$-intersecting if and only if $\ms$ is $t$-intersecting. Moreover, if $\mr$ is $t$-intersecting , then $\mr,\ms\subseteq{Z\c t+1}$ for some $Z\in U^{[ck]}_{c,t+2}$.
\end{lem}
\begin{proof}
	It is sufficient to consider the case that $\mr$ is $t$-intersecting.
	
	Since $\tau_t(\mr)=t+1$, we have $|\mr|\ge3$. Let $R_1, R_2\in\mr$ and $Z=R_1\cup R_2$.
	Observe that there exists a member of $\ms$ not containing $R_1\cap R_2$. 
	It follows from $R_1\cap R_2\not\subseteq R_3$ for some $R_3\in\mr$, and $|R_3\cap R_i|\ge t$ for $i\in\{1,2\}$
	 that $Z\in U^{[ck]}_{c,t+2}$ and $R_3\subseteq Z$.
	We further get $R\in\{R_1,R_2\}$ for each $R\in\mr$ containing $R_1\cap R_2$. Thus $\mr\subseteq{Z\c t+1}$.
	
	Note that each member of $\ms$ contains at least $t$ blocks in $Z$. If $|S\cap Z|=t$ for some $S\in\ms$, then each member of $\mr$ contains $S\cap Z$, a contradiction to the assumption that $\tau_t(\mr)=t+1$. So $\ms\subseteq{Z\c t+1}$, implying that $\ms$ is $t$-intersecting.
\end{proof}

\begin{lem}\label{not-intersecting}
	Let $c$, $k$ and $t$ be positive integers with $c\ge2$ and $k\ge t+3$.
	Suppose $\mr$ and $\ms$ are cross $t$-intersecting subfamilies of $U^{[ck]}_{c,t+1}$ with $\tau_t(\mr)=\tau_t(\ms)=t+1$. If $\mr$ is not $t$-intersecting, then one of the following holds.
	\begin{itemize}
		\item[{\rm(\ro1)}] $t=1$, and  $\mr=\{A_1,A_2,C\}$,  $\ms=\{B_1,B_2,C\}$ for some $A_1=\{e_1,e_2\}$, $A_2=\{e_3,e_4\}$, $B_1=\{e_1,e_3\}$, $B_2=\{e_2,e_4\}$, $C=\{e_1,e_4\}\in U_{c,2}^{[ck]}$.
		\item[{\rm(\ro2)}] $|\mr||\ms|<(t+2)^2$ and $|\mr|+|\ms|\le8$.
	\end{itemize}
\end{lem}
\begin{proof}
	By assumption, we may suppose that $A_1,A_2\in\mr$ with $|A_1\cap A_2|<t$. Pick $S\in\ms$.
	Since $\mr$ and $\ms$ are cross $t$-intersecting, we have $|A_1\cap S|=|A_2\cap S|=t$ and
	$$t+1=|S|\ge|A_1\cap S|+|A_2\cap S|-|A_1\cap A_2\cap S|\ge t+1.$$
	Consequently
	\begin{equation}\label{not-1}
	|A_1\cap A_2\cap S|=|A_1\cap A_2|=t-1,\quad A_1\cap A_2\subseteq S.
	\end{equation}
	
	\medskip
	\noindent{\bf Case 1.} $t=1$
	\medskip
	
	In this case, write
	$A_1=\{e_1,e_2\}$ and $A_2=\{e_3,e_4\}$.
	Notice that $\ms$ is also not $t$-intersecting from Lemma \ref{RS-intersecting}.
	By \eqref{not-1}, we may suppose that $B_1=\{e_1,e_3\},B_2=\{e_2,e_4\}\in\ms$.
	Observe that $\mr\subseteq{B_1\cup B_2\c2}={\{e_1,e_2,e_3,e_4\}\c2}$ and $\ms\subseteq{A_1\cup A_2\c2}={\{e_1,e_2,e_3,e_4\}\c2}$.
	Let $\psi$ denote the maximum value of $r$ such that $2^{\{e_1,e_2,e_3,e_4\}}\cap U^{[ck]}_{c,r}\neq\emptyset$, and $C=\{e_1,e_4\}$, $D=\{e_2,e_3\}$.
	Then $2\le\psi\le4$.
	
	\medskip
	\noindent{\bf Case 1.1.} $\psi=2$
	\medskip
	
	Since $\psi=2$, neither $e_1\cap e_4$ nor $e_2\cap e_3$ is non-empty. Thus $$\{A_1,A_2\}\subseteq\mr\subseteq\{A_1,A_2\},\quad\{B_1,B_2\}\subseteq\ms\subseteq\{B_1,B_2\},$$
	implying that $|\mr|=|\ms|=2$, and (\ro2) follows.
	
		\medskip
	\noindent{\bf Case 1.2.} $\psi=3$
	\medskip
	
	W.l.o.g., assume that $e_1\cap e_4=\emptyset$ and $e_2\cap e_3\neq\emptyset$. Then $C\in U^{[ck]}_{c,2}$.
	Since $\mr$ and $\ms$ are cross $1$-intersecting, we have
	$$\{A_1,A_2\}\subseteq\mr\subseteq\{A_1,A_2,C\},\quad\{B_1,B_2\}\subseteq\ms\subseteq\{B_1,B_2,C\}.$$
	If $\mr=\{A_1,A_2,C\}$ and $\ms=\{B_1,B_2,C\}$, then (\ro1) holds, and  (\ro2) holds otherwise.
	%This together with the maximality of $\mr$ and $\ms$ yields (\ro1).

		\medskip
	\noindent{\bf Case 1.3.} $\psi=4$
	\medskip
	
	In this case, both $C$ and $D$ are in $U^{[ck]}_{c,2}$. Then
	$$\{A_1,A_2\}\subseteq\mr\subseteq\{A_1,A_2,C,D\},\quad\{B_1,B_2\}\subseteq\mr\subseteq\{B_1,B_2,C,D\}.$$
	We have $|\mr|\in\{2,3,4\}$. 
	If $|\mr|=2$ or $|\mr|=4$, then $|\ms|\le 4$ or $|\ms|=2$, respectively, and (\ro2) holds. 
	Suppose $|\mr|=3$. W.l.o.g., assume that $C\in\mr$. Then $D\not\in\ms$.
	We further conclude that (\ro2) follows from $\ms=\{B_1,B_2\}$, and (\ro1) holds if $\ms=\{B_1,B_2,C\}$.

	\medskip
	\noindent{\bf Case 2.} $t=2$
	\medskip
	
	By \eqref{not-1}, let $\{e_0\}=A_1\cap A_2$, and $\mr'=\{A\bs\{e_0\}: A\in\mr\}$, $\ms'=\{B\bs\{e_0\}: B\in\ms\}$.
	We know each member of $\ms$ contains $e_0$. This together with \eqref{not-1} and the assumption that  $\ms$ is  not $t$-intersecting yields $e_0\in R$ for each $R\in\mr$. 
	Consequently $\mr',\ms'\subseteq U^{[ck]}_{c,2}$.
	Since $\mr$ and $\ms$ are cross $2$-intersecting, $\mr'$ and $\ms'$ are cross $1$-intersecting.
	From the discussion in Case 1, we have $|\mr'||\ms'|\le9$ and $|\mr'|+|\ms'|\le6$.
	These together with $|\mr|=|\mr'|$ and $|\ms|=|\ms'|$ yield (\ro2).

	\medskip
	\noindent{\bf Case 3.} $t\ge3$
	\medskip

	From \eqref{not-1}, we get $|\ms|\le4$.
	By Lemma \ref{RS-intersecting}, $\ms$ is also not $t$-intersecting. We further conclude that $|\mr|\le4$, and (\ro2) follows. %This finishes the proof of Lemma \ref{not-intersecting}.
	\end{proof}

\begin{lem}\label{tttt}
	Let $c$, $k$ and $t$ be positive integers with $c\ge6$ and $k\ge t+3$.
	Suppose $\mf$ and $\mg$ are maximal cross $t$-intersecting subfamilies of $\uu$ with $\tau_t(\mf)=\tau_t(\mg)=t+1$ and  $\tau_t(\mt_t(\mf))=\tau_t(\mt_t(\mg))=t+1$.
	If $c\ge4\log_2t+7$ or $k\ge2t+3$, and $|\mf||\mg|\ge(f_2(c,k,t))^2$, then one of the following hold.
	\begin{itemize}
		\item[{\rm(\ro1)}] $\mf=\mg=\mn_2(Z)$ for some $Z\in U^{[ck]}_{c,t+2}$.
		\item[{\rm(\ro2)}] $t=1$, and $\mf=\mn_3(A_1,A_2,C)$ and $\mg=\mn_3(B_1,B_2,C)$, where  $A_1=\{e_1,e_2\},A_2=\{e_3,e_4\},B_1=\{e_1,e_3\},B_2=\{e_2,e_4\},C=\{e_1,e_4\}\in U_{c,2}^{[ck]}$.
	\end{itemize}
\end{lem}
\begin{proof}
	
	We divide our proof into two cases.
	
	\medskip
	\noindent{\bf Case 1.} $\mt_t(\mf)$ is $t$-intersecting.
	\medskip
	
		By Lemma \ref{RS-intersecting}, we have $\mt_t(\mf),\mt_t(\mg)\subseteq{Z\c t+1}$ for some $Z\in U^{[ck]}_{c,t+2}$.
	It follows from $\tau_t(\mt_t(\mf))=t+1$ that $|F\cap Z|\ge t+1$ for each $F\in\mf$, i.e., $\mf\subseteq\mn_2(Z)$.
	Similarly, we also have $\mg\subseteq\mn_2(Z)$.
	Since $\mn_2(Z)$ is $t$-intersecting, 
	by the maximality of $\mf$ and $\mg$,  we conclude that $\mf=\mg=\mn_2(Z)$. Then (\ro1) holds.
	
	\medskip
	\noindent{\bf Case 2.} $\mt_t(\mf)$ is not $t$-intersecting.
	\medskip

	Suppose $|\mt_t(\mf)||\mt_t(\mg)|<(t+2)^2$ and $|\mt_t(\mf)|+|\mt_t(\mg)|\le8$. By Lemma \ref{B}, we have
	$$\dfrac{|\mf||\mg|}{(\tt)^2}\le(t+2)^2-1+\dfrac{12(t+1)(k-t-1)^3}{{(k-t-1)c\c c}}+\dfrac{9(t+1)^2}{4(k-t-1)^{2c-6}}.$$
Then it follows from Lemma \ref{=t+1=t+1} (\ro4) that $|\mf||\mg|<(f_2(c,k,t))^2$, a contradiction. 
Hence by Lemma \ref{not-intersecting}, we have $t=1$ and $\mt_t(\mg)=\{A_1,A_2,C\}$, $\mt_t(\mf)=\{B_1,B_2,C\}$ for some $A_1=\{e_1,e_2\},A_2=\{e_3,e_4\},B_1=\{e_1,e_3\},B_2=\{e_2,e_4\},C=\{e_1,e_4\}\in U_{c,2}^{[ck]}$. 
To get (\ro2), it is sufficient to show that each member of $\mf$ contains at least one element of $\mt_t(\mg)$.

Suppose for contradiction that $F\in\mf$ contains no member of $\mt_t(\mg)$.
Then $F$ contains at most two elements of $\{e_1,e_2,e_3,e_4\}$. On the other hand, by $\mt_t(\mf)=\{B_1,B_2,C\}$, we have $|F\cap\{e_1,e_2,e_3,e_4\}|=2$. Furthermore, $|F\cap\{e_1,e_4\}|=1$.
If $e_1\in F$, then from $B_2\cap F\neq\emptyset$, we obtain $e_2\in F$, a contradiction to $A_1\not\subseteq F$.
If $e_4\in F$, then by $B_1\cap F\neq\emptyset$, we have $e_3\in F$, a contradiction to $A_2\not\subseteq F$.
This finishes our proof.
\end{proof}

	\begin{proof}[Proof of Theorem \ref{HM-perfect}]
		Let $\mf$ and $\mg$ be cross $t$-intersecting subfamilies of $\uu$ such that both $|\bigcap_{F\in\mf}F|$ and $|\bigcap_{G\in\mg}G|$ are less than $t$. Observe that $\min\{\tau_t(\mf),\tau_t(\mg)\}\ge t+1$.
		Suppose $|\mf||\mg|$ takes the maximum value. Then $\mf$ and $\mg$ are maximal.
		By Lemmas \ref{>t+2} \ref{tt}, \ref{ttt} and \ref{tttt}, we conclude that one of Theorem \ref{HM-perfect} (\ro1), (\ro2) or (\ro3) holds. It is routine to check that the family stated in Construction \ref{n3} has size $f_2(c,k,1)$. This together with Lemma \ref{f0f1} finishes our proof.	
	\end{proof}

		\section{Proof of Theorem \ref{SUM}}\label{T13}

		In this section, 
		by constructing an auxiliary bipartite graph and determining all its fragments, 
		we characterize cross $t$-intersecting subfamilies of $\uu$ with maximum sum of their sizes.

	Put $\mx=\my=\uu$. A bipartite graph $G:=G(\mx,\my)$ is defined as follows: for $A\in\mx$ and $B\in\my$, $AB$ is an edge if and only if $|A\cap B|<t$. Observe that $G$ is not complete. 
	It is routine to check that the symmetric group $\Gamma:=S_{ck}$ acts transitively on $\mx$ and $\my$, respectively, in a natural way, and preserves the adjacency relation of $G$. 
	The main result in \cite{SN} shows that the stabilizer $\Gamma_F\cong S_c\wr S_k$ of each vertex $F$ is a maximal subgroup of  $\Gamma$.  
	Then by \cite[Theorem 1.12]{STAB}, the action of $\Gamma$ is \emph{primitive}, i.e., 
	$\Gamma$ preserves no non-trivial partition of $\mx$. 
	A subset $U$ of $\mx$ is said to be \emph{semi-imprimitive} if $1<|U|<|\mx|$ and $|\sigma(U)\cap U|\in\{0,1,|U|\}$ for each $\sigma\in\Gamma$.

	For a subset $\mw$ of the vertices set of $G$, let $N(\mw)$ denote  the set of all vertices $A$ such that $AB$ is an edge of $G$ for some $B\in\mw$. Moreover, if $\mw\cap N(\mw)=\emptyset$, then we say $\mw$ is an \emph{independent set} of $G$, and it is \emph{non-trivial} if $\mw\not\subseteq\mx$ and $\mw\not\subseteq\my$. 
%	Observe that a pair of cross $t$-intersecting families corresponds to a non-trivial independent set in $G$. 

	A \emph{fragment} in $\mx$ is a set $\ma\subseteq\mx$ with 
	$$N(\ma)\neq\my,\quad |N(\ma)|-|\ma|=\min\{|N(\mb)|-|\mb|: \mb\subseteq\mx,\ N(\mb)\neq\my\}.$$ 
		We  can also define the fragments in $\my$. 
		Suppose that $\mf\subsetneq\mx$ and $\mg\subsetneq\my$ are cross $t$-intersecting. 
		They correspond to a non-trivial independent set of $G$ and 
		$$|\mf|+|\mg|\le|\my|-|N(\mf)|+|\mf|.$$
		Consequently, if $|\mf|+|\mg|$ is maximum, then $\mf$ is a fragment in $\mx$.

	  By \cite[Theorem 1.1]{WZ}, the size of each non-trivial independent set of $G$ is at most $|\mx|-|N\{F\}|+1$, where $F\in\mx$, and each fragment in $\mx$ has size $1$ or $|\mx|-|N(\{F\})|$ unless there is a semi-imprimitive fragment in $\mx$ or $\my$. 
	We say a fragment in $\mx$ with size  $1$ or $|\mx|$ is \emph{trivial}, and \emph{non-trivial} otherwise. 
	To prove Theorem \ref{SUM}, it is sufficient to show there is no non-trivial fragment in $\mx$.

	Suppose for contradiction that $\ms\subseteq\mx$ is a minimum sized non-trivial fragment. 
	By \cite[Lemma 2.2]{WZ}, $\ms$ is semi-imprimitive.
	
	\begin{cl}\label{c1}
		There exists no fragment in $\mx$ with size $2$.
	\end{cl}
	\begin{proof}
		Suppose for contradiction that there exists a fragment in $\mx$ with size $2$. Observe that the action of $\Gamma$ on $\mx$ is primitive. 
		Then by \cite[Proposition 2.3]{WZ}, $\Gamma/(\bigcap_{F\in\mx}\Gamma_F)$ is isomorphic to a subgroup of the dihedral group $D_{|\mx|}$. 
		However, $\Gamma/(\bigcap_{F\in\mx}\Gamma_F)\cong\Gamma$ is not isomorphic to any subgroup of $D_{|\mx|}$, a contradiction.
	\end{proof}

	Pick $C\in \ms$. Since the action of $\Gamma$ on $\mx$ is primitive, by \cite[Proposition 2.4]{WZ} and Claim \ref{c1}, there exists unique non-trivial fragment $\mt$ in $\mx$ with $\ms\cap\mt=\{C\}$ and 
	$$|\ms|=|\mt|=\dfrac{1}{2}(|\mx|-|N(\{C\})|+1).$$
	Note that
	$$|N(\ms)|=|N(\mt)|=|N(\{C\})|-1+\dfrac{1}{2}(|\mx|-|N(\{C\})|+1)=\dfrac{1}{2}(|\mx|+|N(\{C\})|-1).$$
	If $N(\ms\cup\mt)=\my$, then 
	$$|N(\{C\})|\le|N(\ms)\cap N(\mt)|=(|\mx|+|N(\{C\})|-1)-|\mx|=|N(\{C\})|-1,$$
	a contradiction. 
	Therefore $N(\ms\cup\mt)\neq\my$. By \cite[Lemma 2.1 (\ro2)]{WZ}, we know $\ms\cup\mt$ is a fragment in $\mx$. 
	Since $\ms\subsetneq(\ms\cup\mt)$, $\ms\cup\mt$ is a trivial fragment and $|\my\bs N(\ms\cup\mt)|=1$. 
	
	\begin{cl}\label{c2}
		$\Gamma_C(\ms\cup\mt)=\ms\cup\mt$.
	\end{cl}
	\begin{proof}
		Observe that $\sigma(\ms)$ is also a nontrivial fragment in $\mx$ containing $C$ for each $\sigma\in \Gamma_C$. 
		If $\sigma(\ms)\neq\ms$, then since $\ms$ is semi-imprimitive and $\ms\cap\sigma(\ms)\neq\emptyset$, we have $|\sigma(\ms)\cap\ms|=1$, i.e., $\sigma(\ms)\cap\ms=\{C\}$, implying that $\sigma(\ms)=\mt$.
		Consequently $\sigma(\ms)\in\{\ms,\mt\}$.  So is $\sigma(\mt)$. 
		Notice that $\ms\subseteq\Gamma_C(\ms)$ and $\mt\subseteq\Gamma_C(\mt)$. 
		We further conclude that 
		$\Gamma_C(\ms\cup\mt)=\Gamma_C(\ms)\cup \Gamma_C(\mt)=\ms\cup\mt$.
	\end{proof}

	\begin{cl}\label{c3}
		$\ms$ is $t$-intersecting.
	\end{cl}
	\begin{proof}
		Let $Z$ be the unique member of $\my\bs N(\ms\cup\mt)$, and $A\in\ms\cup\mt$. Next we show 
		there exists $B\in\ms\cup\mt$ such that $A\cap C\cap Z=B\cap Z$. 
		
		If $A\cap C\cap Z=A\cap Z$, then there is nothing to prove. Now suppose $A\cap C\cap Z\neq A\cap Z$. 
		Then $e\in A\cap Z$ and $e\not\in C$ for some $e\in A$. 
		Notice that there exist at least two blocks in $C$ intersecting $e$. 
		Furthermore, we have $e\cap h\neq\emptyset$ and $f\cap h\neq\emptyset$ for some $f\in A\bs\{e\}$ and $h\in C$. 
		
		Pick $i\in e\cap h$ and $j\in f\cap h$. Set $\sigma=(i\ j)$. 
		We have $\sigma\in\Gamma_C$, and $\sigma(A)=(A\bs\{e,f\})\cup\{\sigma(e),\sigma(f)\}\in\ms\cup\mt$ by Claim \ref{c2}. 
		Observe that $\sigma(e),\sigma(f)\not\in Z$. 
		Then $\sigma(A)\cap Z\subsetneq A\cap Z$ and $\sigma(A)\cap C\cap Z=A\cap C\cap Z$. 
		If $\sigma(A)\cap Z=A\cap C\cap Z$, then let $B=\sigma(A)$. 
		If $\sigma(A)\cap Z\neq A\cap C\cap Z$, then do a similar operation on $\sigma(A)$. 
		Since $(A\cap Z)\bs C$ is finite, we finally get $B\in\ms\cup\mt$ with $A\cap C\cap Z=B\cap Z$.
		
		Now $|A\cap C|\ge |A\cap C\cap Z|=|B\cap Z|\ge t$. 
		By the arbitrariness of the selection of $C$, we know $\ms$ is $t$-intersecting.	
	\end{proof}

	Let $$N_i(C)=\{A\in\uu:|A\cap C|=i\}.$$	
	Notice that for $i\in\{t,t+1,\dots,k-1\}$,
	\begin{equation}\label{theta-C}
		\theta(c,k,i)=\sum_{j=i}^{k-2}\dfrac{{k-i\c j-i}}{{k\c j}}|N_j(C)|+1.
	\end{equation}
	By Theorem \ref{EKR-perfect} and Claim \ref{c3}, we have
	$$|\mx|-|N(\{C\})|+1=2|\ms|\le2\theta(c,k,t)=2\left(\dfrac{1}{{k\choose t}}|N_t(C)|+\sum_{j=t+1}^{k-2}\dfrac{{k-t\c j-t}}{{k\c j}}|N_j(C)|+1\right).$$
	Hence
	$$\left(1-\dfrac{2}{{k\c t}}\right)|N_t(C)|\le\sum_{j=t+1}^{k-2}\left(\dfrac{2{k-t\c j-t}}{{k\c j}}-1\right)|N_j(C)|.$$
	This together with \eqref{theta-C} produces 
	\begin{equation*}
		\begin{aligned}
			%\dfrac{\theta(c,k,t)-(k-t)\theta(c,k,t+1)}{2}
			\theta(c,k,t)-(k-t)\theta(c,k,t+1)
			&\le\dfrac{1}{{k\c t}}|N_t(C)|-\sum_{j=t+1}^{k-2}\dfrac{(j-t-1){k-t\c j-t}}{{k\c j}}+1-(k-t)\\
			&\le\left(1-\dfrac{2}{{k\choose t}}\right)|N_t(C)|-\sum_{j=t+1}^{k-2}\dfrac{(j-t-1){k-t\c j-t}}{{k\c j}}|N_j(C)|\\
			&\le\sum_{j=t+1}^{k-2}\dfrac{(3-j+t){k-t\c j-t}-{k\c j}}{{k\c j}}|N_j(C)|.
		\end{aligned}
	\end{equation*}
	Notice that, if $k\ge t+3$, then $2\le t+1\le k-2$ and
	$$2{k-t\c1}-{k\c t+1}\le2(k-t)-{k\c2}\le0;$$
	if $k\ge t+4$, then $3\le t+2\le k-2$ and 
	$${k-t\c2}-{k\c t+2}\le{k-t\c2}-\min\left\{{k\c3},{k\c k-2}\right\}\le0.$$
	Then 
	\begin{equation*}
		\begin{aligned}
	2&\le(k-t)((k-t)^{c-2}-1)=(k-t)^{c-1}-(k-t)\\
	&<\dfrac{1}{k-t}{(k-t)c\c c}-(k-t)
	=\dfrac{\theta(c,k,t)-(k-t)\theta(c,k,t+1)}{\theta(c,k,t+1)}\le0,
		\end{aligned}
\end{equation*}
	a contradiction. Consequently, there exists no non-trivial fragment in $\mx$. This finishes the proof of Theorem \ref{SUM}.
	\qed

		\section{The case $k=t+2$}\label{t+2}
		
		Theorem \ref{HM-perfect} addresses the case $k\ge t+3$. In this section, we characterize all maximal cross $t$-intersecting subfamilies of $\uu$ for $k=t+2$, and for a finite set $X$, let $U^{X}_{c,\ell}$ denote the set of all families consisting of $\ell$ pairwise disjoint $c$-subsets of $X$. 
		 
	\begin{con}\label{ex-t+2-1}
		Suppose $c$, $k$ and $t$ are positive integers with $c\geq2$ and $k=t+2$. Let
		$$\mf=\{\{e_1,\dots,e_{t+2}\},\{e_1''',e_2,\dots,e_t,e_{t+1}',e_{t+2}''\}\}$$
		and
		$$\mg=\{\{e_1,\dots,e_t,e_{t+1}',e_{t+2}',\},\{e_1'',e_2,\dots,e_{t+1},e_{t+2}''\}\}$$
		be two subfamilies of $\uu$, where $e_{t+1}'\in{e_{t+1}\cup e_{t+2}\choose c}\bs\{e_{t+1},e_{t+2}\}$, $e_{t+2}''\in{e_1\cup e_{t+2}\choose c}\bs\{e_1,e_{t+2}\}$, and $e_{t+2}\not\subseteq e_{t+1}'\cup e_{t+2}''$.
	\end{con}

	\begin{con}\label{ex-t+2-2}
		Suppose $c$, $k$ and $t$ are positive integers with $c\geq2$ and $k=t+2$. Let
		$$\mf=\{\{e_1,\dots,e_{t+2}\},\{e_1''',e_2,\dots,e_t,e_{t+1}',e_{t+2}''\},\{e_1'',e_2,\dots,e_t,e_{t+1}''',e_{t+2}'\}\}$$
		and
		$$\mg=\{\{e_1,\dots,e_t,e_{t+1}',e_{t+2}'\},\{e_1'',e_2,\dots,e_{t+1},e_{t+2}''\},\{e_1''',e_2,\dots,e_t,e_{t+1}''',e_{t+2}\}\}$$
		be two subfamilies of $\uu$, where $e_{t+1}'\in{e_{t+1}\cup e_{t+2}\choose c}\bs\{e_{t+1},e_{t+2}\}$, $e_{t+2}''\in{e_1\cup e_{t+2}\choose c}\bs\{e_1,e_{t+2}\}$, and $e_{t+2}\subseteq e_{t+1}'\cup e_{t+2}''$.
	\end{con}

	\begin{con}\label{ex-t+2-3}
		Suppose $c$, $k$ and $t$ are positive integers with $c\geq2$ and $k=t+2$. Let
		$$\mf=\{\{e_1,\dots,e_{t+2}\},\{e_1',e_2',e_3,\dots,e_t,e_{t+1}',e_{t+2}'\}\}$$
		and
		$$\mg=\{\{e_1,\dots,e_t,e_{t+1}',e_{t+2}'\},\{e_1',e_2',e_3\dots,e_{t+2}\}\}$$
		be two subfamilies of $\uu$, where $\{e_1',e_2'\}\in U^{e_1\cup e_2}_{c,2}\bs\{ \{ e_1,e_2\} \}$ and $\{e_{t+1}',e_{t+2}'\}\in U^{e_{t+1}\cup e_{t+2}}_{c,2}\bs\{ \{ e_{t+1},e_{t+2}\} \}$.
	\end{con}

 For maximal cross $t$-intersecting subfamilies $\mf$ and $\mg$ of $U^{[c(t+2)]}_{c,t+2}$, if $\tau_t(\mf)=\tau_t(\mg)=t$, then by Lemma \ref{cover-intersection}, we know $\mf=\mg=\{F\in U^{[c(t+2)]}_{c,t+2}: T\subseteq F\}$ for some $T\in U^{[c(t+2)]}_{c,t}$. Therefore, in the following theorem, we may assume that $\tau_t(\mg)\ge t+1$.

	\begin{thm}\label{thm-t+2}
		Let $c$, $k$ and $t$ be positive integers with $c\ge 2$. Suppose $\mf$ and $\mg$ are  maximal cross $t$-intersecting subfamilies of $\uu$. If $k=t+2$ and $\tau_t(\mg)\ge t+1$, then one of the following holds.
		\begin{itemize}
			\item[\rm{(\ro1)}] $\mf=\{F\}$ and $\mg=\{G\in \uu: |G\cap F|\ge t\}$ for some $F\in \uu$.
			\item[\rm{(\ro2)}] $\mf$ and $\mg$ are families stated in Constructions \ref{ex-t+2-1} or \ref{ex-t+2-2}.
			\item[\rm{(\ro3)}] $t\ge2$ and $\mf$ and $\mg$ are families stated in Construction \ref{ex-t+2-3}.
			
		\end{itemize}
	\end{thm}
	\begin{proof}
		Pick $R:=\{e_1,\dots,e_{t+2}\}\in U^{[c(t+2)]}_{c,t+2}$. 
		W.l.o.g., let $S=\{e_1,e_2,\dots,e_{\tau_t(\mg)}\}\in\mt_t(\mg)$.
		If $|G\cap S|\ge t+1$ for $G\in\mg$, then $S\subseteq G$. Since $\tau_t(\mg)\ge t+1$, we have $|G_1\cap S|=t$ for some $G_1\in\mg$. Write $T_1=G_1\cap S$. W.l.o.g., assume that $T_1=\{e_1,\dots,e_t\}$.
		Observe that $T_1\not\subseteq G_2$ for some $G_2\in\mg$. Then we also have $|G_2\cap S|=t$. Write $T_2=G_2\cap S$ and
		$$G_1=\{e_1,\dots,e_t,e_{t+1}',e_{t+2}'\},$$ where $\{e_{t+1}',e_{t+2}'\}\in U_{c,2}^{e_{t+1}\cup e_{t+2}}\bs\{\{e_{t+1},e_{t+2}\}\}$.
		Notice that $T_1\neq T_2$. We may assume $e_1\not\in G_2$ and $e_{t+1}\in G_2$.

	Since $S$ is a $t$-cover of $\mg$, by the maximality of $\mf$ and $\mg$, we have $R\in\mf$. Therefore, if $|\mf|=1$, then (\ro1) holds.
		Next assume that $\mf\bs\{R\}\neq\emptyset$.

		\begin{cl}\label{cl1}
			If $F\in\mf\bs\{R\}$, then $e_1\not\in F$, $T_1\not\subseteq F$, and $e_{t+1}\not\in F$, $T_2\not\subseteq F$. Moreover, $F\cap\{e_{t+1}',e_{t+2}'\}\neq\emptyset$.
		\end{cl}
		\begin{proof}
			Suppose for contradiction that $e_1\in F$. Since $e_1\not\in G_2$, there exist at least two blocks in $G_2$ intersecting $e_1$. Notice that these blocks are not in $T_2$.
			By $e_1\in F$ and $|F\cap G_2|\ge t$, we have $T_2\subseteq F$, implying that $S\subseteq F$ and $F=R$, a contradiction. Thus  $e_1\not\in F$ and $T_1\not\subseteq F$. Similarly, we have $e_{t+1}\not\in F$ and $T_2\not\subseteq F$.
			Then it follows from $|F\cap G_1|\ge t$ that $F\cap\{e_{t+1}',e_{t+2}'\}\neq\emptyset$.
		\end{proof}

	 Notice that $|T_1\cap G_2|=|T_1\cap S\cap G_2|=|T_1\cap T_2|\in\{t-2,t-1\}$.

		\begin{cl}
			Suppose $|T_1\cap G_2|=t-1$. Then $\mf$ and $\mg$ are families stated in Constructions \ref{ex-t+2-1} or \ref{ex-t+2-2}.
		\end{cl}
		\begin{proof}
			In this case, we have $T_2=\{e_2,\dots,e_{t+1}\}$ and
			$$G_2=\{e_1'',e_2,\dots,e_{t+1},e_{t+2}''\},$$ where $\{e_1'',e_{t+2}''\}\in U_{c,2}^{e_{1}\cup e_{t+2}}\bs\{\{e_{1},e_{t+2}\}\}$.

			Let $F\in\mf\bs\{R\}$. We first show $|F\cap\{e_{t+1}',e_{t+2}'\}|=|\{F\cap\{e_1'',e_{t+2}''\}\}|=1$.
			If $\{e_{t+1}',e_{t+2}'\}\subseteq F$, then $e_{t+1},e_1'',e_{t+2}''\not\in F$, implying that $|F\cap G_2|\le t-1$,  a contradiction to the assumption that $\mf$ and $\mg$ are cross $t$-intersecting.
			Notice that $|F\cap G_1|\ge t$ and  $T_1\not\subseteq F$. 
			We know $|F\cap\{e_{t+1}',e_{t+2}'\}|=1$. 
			On the other hand, by $e_{t+2}\cap e_{t+1}'\neq\emptyset$,  $e_{t+2}\cap e_{t+2}'\neq\emptyset$ and $e_{t+2}\subseteq e_1''\cup e_{t+2}''$, we have $|\{F\cap\{e_1'',e_{t+2}''\}\}|\le1$.
			This together with $|F\cap G_2|\ge t$ and $T_2\not\subseteq F$ produces
			$|\{F\cap\{e_1'',e_{t+2}''\}\}|=1$, as desired.
			
			Therefore, w.l.o.g., we may suppose that $e_{t+1}'\cap e_{t+2}''=\emptyset$.
			Observe that at least one of $e_{t+1}'\cap e_{t+2}''\cap e_{t+2}$ and $e_{t+2}'\cap e_{1}''\cap e_{t+2}$ is empty, and both $e_{t+1}'\cap e_1''$ and $e_{t+2}'\cap e_{t+2}''$ are non-empty. 
			We have $\{e_{t+1}',e_{t+2}''\}\subseteq F$ or $\{e_1'',e_{t+2}'\}\subseteq F$. Pick
			$$F_1=\{e_1''',e_2,\dots,e_t,e_{t+1}',e_{t+2}''\}\in\uu.$$
			It follows from $|F\cap G_1|\ge t$, $e_1\not\in F$ and $|F\cap\{e_{t+1}',e_{t+2}'\}|=1$ that $e_2,\dots,e_t\in F$.
			Furthermore, if $\{e_{t+1}',e_{t+2}''\}\subseteq F$, then $F=F_1$.

			Let $G\in\mg$. Recall that $\{e_{t+1}',e_{t+2}''\}\subseteq F$ or $\{e_1'',e_{t+2}'\}\subseteq F$.
			If $e_1\in G$, then by $|F\cap G|\ge t$ and $e_1\not\in F$, we have
			$e_2,\dots,e_t\in G$ and $G\cap\{e_{t+1}',e_{t+2}'\}\neq\emptyset$.
			We further conclude that $G=G_1$. Similarly, if $e_{t+1}\in G$, then $G=G_2$. Suppose $G\cap\{e_1,e_{t+1}\}=\emptyset$. By $|G\cap R|\ge t$, we have $e_2,\dots,e_t,e_{t+2}\in G$.

			Assume that $e_{t+2}\bs(e_{t+1}'\cup e_{t+2}'')\neq\emptyset$.
			Then this non-empty set is contained in $e_{1}''\cap e_{t+2}'$.
			So no member of $\mf\bs\{R\}$ contains $\{e_1'',e_{t+2}'\}$, implying that  $\mf\subseteq\{R,F_1\}$. 
			Then  $\mf=\{R,F_1\}$ follows from $|\mf|\ge2$.
			Since $e_1'''$, $e_{t+1}'$ and $e_{t+2}''$ intersect $e_{t+2}$, each member of $\mg$ contains $e_1$ or $e_{t+1}$, implying that
			$\mg=\{G_1,G_2\}$.
			Then $\mf$ and $\mg$ are families stated in Construction \ref{ex-t+2-1}.

			Assume that $e_{t+2}\bs(e_{t+1}'\cup e_{t+2}'')=\emptyset$.
			Then $e_1''\cap e_{t+2}'=\emptyset$. Set
			$$F_2=\{e_1'',e_2,\dots,e_t,e_{t+1}''',e_{t+2}'\}\in\uu.$$
			Recall that $F\in\mf\bs\{R\}$. If $\{e_1'',e_{t+2}'\}\subseteq F$, by
			$|F\cap G_2|\ge t$ and $e_{t+2}'\cap e_{t+2}=e_{t+2}''\cap e_{t+2}$,
			we have $F=F_2$. Therefore
			$\mf\subseteq\{R,F_1,F_2\}$.
			Set
			$$G_3=\{e_1''',e_2,\dots,e_t,e_{t+1}''',e_{t+2}\}.$$
			It is routine to check that $G_3\in\uu$.
			Recall that $G\in\mg$ with $G\cap\{e_1,e_{t+1}\}=\emptyset$ contains $e_2,\dots,e_t,e_{t+2}$.
			Since $|\mf|\ge2$, at least one of $|G\cap F_1|$ and $|G\cap F_2|$ is no less than $t$.
			Then $e_1'''\in G$ or $e_{t+1}'''\in G$, implying that $G=G_3$.
			Hence $\mg\subseteq\{G_1,G_2,G_3\}$. Notice that families $\{R,F_1,F_2\}$ and $\{G_1,G_2,G_3\}$ are cross $t$-intersecting. By the maximality of $\mf$ and $\mg$, we have $\mf=\{R,F_1,F_2\}$ and $\mg=\{G_1,G_2,G_3\}$, as the families stated in Construction \ref{ex-t+2-2}.
		\end{proof}

		\begin{cl}
			Suppose $|T_1\cap G_2|=t-2$. Then $\mf$ and $\mg$ are families stated in Construction \ref{ex-t+2-3}.
		\end{cl}
		\begin{proof}
			In this case, we have $\tau_t(\mg)=t+2$ and $S=R$.
			W.l.o.g., assume $G_2\cap S=T_2=\{e_3,\dots,e_{t+2}\}$, and write
			$$G_1=\{e_1,\dots,e_t,e_{t+1}',e_{t+2}'\},\quad G_2=\{e_1',e_2',e_3\dots,e_{t+2}\},$$
			where $\{e_1',e_2'\}\in U^{e_1\cup e_2}_{c,2}\bs\{ \{ e_1,e_2\} \}$ and $\{e_{t+1}',e_{t+2}'\}\in U^{e_{t+1}\cup e_{t+2}}_{c,2}\bs\{ \{ e_{t+1},e_{t+2}\} \}$.
			Let $F\in\mf\bs\{S\}$.
			Note that if $e_i\in F$ for some $i\in\{1,2\}$, then by $|F\cap G_2|\ge t$, we have $e_1',e_2'\not\in F$ and $T_2\subseteq F$, implying that $S=F$. Hence  $e_1,e_2\not\in F$.
			We further conclude that $F=\{e_1',e_2',e_3,\dots,e_t,e_{t+1}',e_{t+2}'\}$ and $\mf=\{S,F\}$.
			
			Let $G\in\mg$. Observe that $\{e_1,e_2,e_{t+1},e_{t+2}\}\cap G\neq\emptyset$.
			This together with $|F\cap G|\ge t$ yields $\{e_3,\dots,e_t,e_{t+1}',e_{t+2}'\}\subseteq G$ or $\{e_1',e_2',e_3,\dots,e_t\}\subseteq G$.
			Furthermore, $G=G_1$ or $G=G_2$ from $|G\cap S|\ge t$. Then $\mg=\{G_1,G_2\}$, and $\mf$ and $\mg$ are families stated in Construction \ref{ex-t+2-3}.
		\end{proof}
		Now we finish the proof  Theorem \ref{thm-t+2}.
	\end{proof}

\section{Inequalities}\label{COMPUTION}

In this section, we prove some inequalities used in this paper. 

\begin{lem}\label{key-bidaxiao}
	Let $c$, $k$ and $t$ be positive integers with $c\ge3$ and $k\ge t+2$.	If $c\ge3+2\log_2t$ or $k\ge 2t+2$ and  then the following hold.
	\begin{itemize}
		\item[{\rm(\ro1)}] 	$g(c,k,t,s+1)<g(c,k,t,s)$ for each $s\in\{t,t+1,\dots,k-2\}$.
		\item[{\rm(\ro2)}] 	$g(c,k,t,k)<g(c,k,t,k-2)$.
	\end{itemize}
\end{lem}
\begin{proof}
	(\ro1) Let $s\in\{t,t+1,\dots,k-2\}$.
	By ${(k-s)c\choose c}>(k-s)^c$ for $c\ge3$, we have
	$$\dfrac{g(c,k,t,s+1)}{g(c,k,t,s)}
	=\dfrac{(k-s)^2}{{(k-s)c\c c}}\cdot\dfrac{s+1}{s-t+1}
	<\dfrac{1}{(k-s)^{c-2}}\cdot\dfrac{s+1}{s-t+1}.$$
	Suppose $k\ge2t+2$. If $s\ge\frac{k}{2}$, then
	\begin{align*}
		\dfrac{g(c,k,t,s+1)}{g(c,k,t,s)}&<\dfrac{1}{k-s}\cdot\left(1+\dfrac{t}{s-t+1}\right)\le\dfrac{1}{2}+\max\left\{\dfrac{1}{\frac{k}{2}}\cdot\dfrac{t}{\frac{k}{2}-t+1},\dfrac{1}{2}\cdot\dfrac{t}{k-t-1}\right\}\\
		&\le\dfrac{1}{2}+\dfrac{1}{2}\cdot\dfrac{t}{t+1}<1;
	\end{align*}
	If $s<\frac{k}{2}$, then
	$$\dfrac{g(c,k,t,s+1)}{g(c,k,t,s)}<\dfrac{1}{\frac{k}{2}}\cdot(t+1)\le1.$$
	Next suppose $c\ge3+2\log_2t$. By $t\le s\le k-2$, we have
	$$\dfrac{g(c,k,t,s+1)}{g(c,k,t,s)}<\dfrac{t+1}{2^{c-2}}\le\dfrac{t+1}{2t^2}\le1.$$

	(\ro2) Observe that
	$$\dfrac{g(c,k,t,k)}{g(c,k,t,k-2)}=\dfrac{4}{{2c\choose c}}\cdot\dfrac{(k-1)k}{(k-t-1)(k-t)}=\dfrac{4}{{2c\c c}}\left(1+\dfrac{t}{k-t-1}\right)\left(1+\dfrac{t}{k-t}\right).$$
	Suppose $k\ge2t+2$. By $c\ge3$, we have
	$$\dfrac{g(c,k,t,k)}{g(c,k,t,k-2)}\le\dfrac{1}{5}\left(1+\dfrac{t}{t+1}\right)\left(1+\dfrac{t}{t+2}\right)<\dfrac{4}{5}<1.$$
	Next suppose $c\ge3+2\log_2t$. We have
	$$\dfrac{g(c,k,t,k)}{g(c,k,t,k-2)}\le\dfrac{12}{{6\choose3}}<1$$
	for $t=1$, and
	$$\dfrac{g(c,k,t,k)}{g(c,k,t,k-2)}<\dfrac{1}{2^{c-2}}\left(t+1\right)\left(\dfrac{t}{2}+1\right)\le\dfrac{3t^2}{2^{c-1}}<1$$
	for $t\ge2$, as desired.
	%From $k\ge t+3$, we get
%$$\dfrac{g(c,k,t,k)}{g(c,k,t,k-2)}<\dfrac{1}{2^{c-2}}\left(1+\dfrac{t}{2}\right)\left(1+\dfrac{t}{3}\right)\le\dfrac{t^2}{2^{c-3}}\le1,$$
	%as desired.
\end{proof}

\begin{lem}\label{=t+1=t+1}
	Let $c$, $k$ and $t$ be  positive integers with $c\ge6$ and $k\ge t+3$. If $c\ge4\log_2t+7$ or $k\ge2t+3$, then the following hold.
	\begin{itemize}
		\item[{\rm(\ro1)}] $g(c,k,t,t+1)g(c,k,t,t+2)<(f_0(c,k,t))^2$.
		\item[{\rm(\ro2)}] $g(c,t+3,t,t+1)g(c,t+3,t,t+3)<(f_0(c,t+3,t))^2$.
	\end{itemize}
\end{lem}

\begin{proof}
(\ro1) Note that
	\begin{equation}\label{f/g}
		\begin{aligned}
			\dfrac{f_0(c,k,t)}{g(c,k,t,t+1)}&=\dfrac{k-t-1}{(t+1)(k-t)}\left(1-\dfrac{(k-t-1)(k-t-2)}{2{(k-t-1)c\c c}}\right),\\
			\dfrac{f_0(c,k,t)}{g(c,k,t,t+2)}&=\dfrac{1}{(t+1)(t+2)}\left(\dfrac{2}{(k-t)(k-t-1)}{(k-t-1)c\c c}-\dfrac{k-t-2}{k-t}\right).
		\end{aligned}
	\end{equation}
	If $c\ge4\log_2t+7$, then  by $k\ge t+3$, we have
	\begin{equation*}
		\begin{aligned}
			\dfrac{f_0(c,k,t)}{g(c,k,t,t+1)}&>\dfrac{1}{3t}\left(1-\dfrac{1}{2(k-t-1)^{c-2}}\right)
			\ge\dfrac{1}{3t}\left(1-\dfrac{1}{64t^4}\right)\ge\dfrac{21}{64t},\\
			\dfrac{f_0(c,k,t)}{g(c,k,t,t+2)}&>\dfrac{1}{(t+1)(t+2)}\left(\dfrac{4}{3}(k-t-1)^{c-2}-1\right)\ge\dfrac{\frac{128}{3}t^4-1}{6t^2}\ge\dfrac{125}{18}t^2.
		\end{aligned}
	\end{equation*}
	We further get the desired result.
	If $k\ge 2t+3$, then  by $c\ge6$, we have
	\begin{equation*}
		\begin{aligned}
			\dfrac{f_0(c,k,t)}{g(c,k,t,t+1)}&>\dfrac{3}{4(t+1)}\left(1-\dfrac{1}{2(k-t-1)^4}\right)\ge\dfrac{3}{4(t+1)}\left(1-\dfrac{1}{2(t+2)^4}\right)\ge\dfrac{161}{216(t+1)},\\
			\dfrac{f_0(c,k,t)}{g(c,k,t,t+2)}&>\dfrac{\frac{4}{3}(t+2)^4-1}{(t+1)(t+2)}\ge\dfrac{4}{3}(t+2)^2.
		\end{aligned}
	\end{equation*}
	We also conclude that  $g(c,k,t,t+1)g(c,k,t,t+2)<(f_0(c,k,t))^2$.

	(\ro2) Notice that $c\ge4\log_2t+7$ when $k=t+3$. Then
	$$\dfrac{f_0(c,t+3,t)}{g(c,t+3,t,t+1)}=\dfrac{2}{3(t+1)}\left(1-\dfrac{1}{{2c\c c}}\right)\ge\dfrac{1}{3t}\left(1-\dfrac{1}{128t^4}\right)\ge\dfrac{127}{384t}$$
	and
	$$\dfrac{f_0(c,t+3,t)}{g(c,t+3,t,t+3)}=\dfrac{{2c\c c}-1}{6{t+3\c3}}>\dfrac{2^{c}-1}{24t^3}\ge\dfrac{128t^4-1}{24t^3}\ge\dfrac{127}{24}t$$
	yield the desired result.
\end{proof}

Write
\begin{equation*}
	\begin{aligned}
	h_1(c,k,t)&=\left((t+1)(k-t-1)+\dfrac{3(t+1)}{2(k-t-1)^{c-3}}\right)\left(1+\dfrac{3(t+1)}{2(k-t-1)^{c-3}}\right)(\tt)^2,\\
	h_2(c,k,t)&=\left((k-1)+\dfrac{3(t+1)}{2(k-t-1)^{c-3}}\right)\left(2+\dfrac{3(t+1)}{2(k-t-1)^{c-3}}\right)(\tt)^2,\\
	h_3(c,k,t)&=\left(k-t-1+\dfrac{3(t+1)}{2(k-t-1)^{c-3}}\right)\left(k-t-2+\dfrac{3(t+1)}{2(k-t-1)^{c-3}}\right)(\tt)^2,\\
	h_4(c,k,t)&=\left((t+2)^2-1+\dfrac{12(t+1)(k-t-1)^3}{{(k-t-1)c\c c}}+\dfrac{9(t+1)^2}{4(k-t-1)^{2c-6}}\right)(\tt)^2,
	\end{aligned}
\end{equation*}
	Observe that
\begin{equation}\label{f2}
	\begin{aligned}
		\dfrac{f_2(c,k,t)}{\tt}&=(t+2)-\dfrac{(t+1)(k-t-1)}{{(k-t-1)c\c c}}.
	\end{aligned}
\end{equation}

%$$h_1(c,k,t)=\left((t+1)(k-t-1)+\dfrac{3(t+1)}{2(k-t-1)^{c-3}}\right)\left(1+\dfrac{3(t+1)}{2(k-t-1)^{c-3}}\right)(\tt)^2,$$
%$$h_2(c,k,t)=\left((k-1)+\dfrac{3(t+1)}{2(k-t-1)^{c-3}}\right)\left((2+\dfrac{3(t+1)}{2(k-t-1)^{c-3}}\right)(\tt)^2,$$
%$$h_3(c,k,t)=\left(k-t-1+\dfrac{3(t+1)}{2(k-t-1)^{c-3}}\right)\left(k-t-2+\dfrac{3(t+1)}{2(k-t-1)^{c-3}}\right)(\tt)^2,$$
%$$h_4(c,k,t)=\left((t+2)^2+\dfrac{12(t+1)}{(k-t-1)^{c-3}}+\dfrac{9(t+1)^2}{4(k-t-1)^{2c-6}}\right)(\tt)^2.$$
\begin{lem}\label{=t=t}
		Let $c$, $k$ and $t$ be  positive integers with $c\ge6$ and $k\ge t+3$. If $c\ge4\log_2t+7$ or $k\ge2t+3$, then the following hold.
	\begin{itemize}
		\item[{\rm(\ro1)}] $h_1(c,k,t)<\max\{(f_0(c,k,t))^2,(f_2(c,k,t))^2\}$.
		\item[{\rm(\ro2)}] $h_2(c,k,t)<\max\{(f_0(c,k,t))^2,(f_2(c,k,t))^2\}$.
		\item[{\rm(\ro3)}] $h_3(c,k,t)<(f_0(c,k,t))^2$.
		\item[{\rm(\ro4)}] $h_4(c,k,t)<(f_2(c,k,t))^2$.
	\end{itemize}
\end{lem}

\begin{proof}
	By \eqref{f0} and \eqref{f2}, we have
	\begin{equation}\label{f0f2-b}
		\begin{aligned}
			\dfrac{f_0(c,k,t)}{\tt}\ge k-t-1-\dfrac{1}{2(k-t-1)^{c-3}},\quad
			\dfrac{f_2(c,k,t)}{\tt}\ge t+2-\dfrac{t+1}{(k-t-1)^{c-1}}.
		\end{aligned}
	\end{equation}

(\ro1) 	By \eqref{f0f2-b}, we have
	\begin{equation*}
		\begin{aligned}
			\dfrac{(f_0(c,k,t))^2-h_1(c,k,t)}{((k-t-1)\tt)^2}&>1-\dfrac{t+1}{k-t-1}-\dfrac{1}{(k-t-1)^4}-\dfrac{3(t+1)(t+2)}{2(k-t-1)^4}-\dfrac{9(t+1)^2}{4(k-t-1)^8}\\
			&>\dfrac{1}{t+2}-\dfrac{1}{(t+2)^4}-\dfrac{3}{2(t+2)^2}-\dfrac{9}{4(t+2)^6}>0
		\end{aligned}
	\end{equation*}	
	for $k\ge 2t+3$,
	and
	\begin{equation*}
		\begin{aligned}
			\dfrac{(f_2(c,k,t))^2-h_1(c,k,t)}{((t+2)\tt)^2}>&1-\dfrac{(t+1)(k-t-1)}{(t+2)^{2}}-\dfrac{2(t+1)}{(t+2)(k-t-1)^{c-1}}\\
			&-\dfrac{3(t+1)}{2(t+2)(k-t-1)^{c-4}}-\dfrac{9(t+1)^2}{4(t+2)^2(k-t-1)^{2c-6}}\\
			\ge&\dfrac{2t+3}{(t+2)^2}-\dfrac{t+1}{32t^4(t+2)}-\dfrac{3(t+1)}{16t^4(t+2)}-\dfrac{9(t+1)^2}{1024t^8(t+2)^2}\\
			\ge&\dfrac{1}{t+2}\left(\dfrac{5}{3}-\dfrac{1}{16t^3}-\dfrac{3}{8t^3}-\dfrac{9}{256t^6}\right)>0	
		\end{aligned}
	\end{equation*}	
	for $k\le2t+2$.  Then the desired result follows.
	
	(\ro2)
	By \eqref{f0f2-b}, we have
	\begin{equation*}
		\begin{aligned}
			\dfrac{(f_0(c,k,t))^2-h_2(c,k,t)}{((k-t-1)\tt)^2}
			>&1-\dfrac{2(k-1)}{(k-t-1)^2}-\dfrac{1}{(k-t-1)^4}-\dfrac{3(t+1)(k+1)}{2(k-t-1)^5}\\
			&-\dfrac{9(t+1)^2}{4(k-t-1)^8}\\
			\ge&1-\dfrac{4(t+1)}{(t+2)^2}-\dfrac{1}{(t+2)^4}-\dfrac{3(t+1)}{(t+2)^4}-\dfrac{9(t+1)^2}{4(t+2)^8}\\
			\ge&1-\dfrac{8}{9}-\dfrac{1}{81}-\dfrac{2}{27}-\dfrac{1}{729}>0
		\end{aligned}
	\end{equation*}
	for $k\ge 2t+3$,
	and
	\begin{equation*}
		\begin{aligned}
			\dfrac{(f_2(c,k,t))^2-h_2(c,k,t)}{((t+2)\tt)^2}>&1-\dfrac{2(k-1)}{(t+2)^2}-\dfrac{2(t+1)}{(t+2)(k-t-1)^{c-1}}-\dfrac{3(t+1)(k+1)}{2(t+2)^2(k-t-1)^{c-3}}\\
			&-\dfrac{9(t+1)^2}{4(t+2)^2(k-t-1)^{2c-6}}
			\\
			>&1-\dfrac{2(2t+1)}{(t+2)^2}-\dfrac{1}{32t^4}-\dfrac{3}{16t^4}-\dfrac{9}{1024t^8}>0
		\end{aligned}
	\end{equation*}
	for $k\le2t+2$. Then the desired result follows.

(\ro3)
	By \eqref{f0f2-b}, we have
	\begin{equation*}
		\begin{aligned}
			\dfrac{(f_0(c,k,t))^2-h_3(c,k,t)}{((k-t-1)\tt)^2}>&1-\dfrac{k-t-2}{k-t-1}-\dfrac{1}{(k-t-1)^{c-2}}-\dfrac{3(t+1)(2k-2t-3)}{2(k-t-1)^{c-1}}\\
			&-\dfrac{9(t+1)^2}{4(k-t-1)^{2c-4}}\ge\dfrac{\varphi(c,k,t)}{k-t-1},
		\end{aligned}
	\end{equation*}
	where $\varphi(c,k,t)=1-\frac{1}{(k-t-1)^{c-3}}-\frac{3(t+1)}{(k-t-1)^{c-3}}-\frac{9(t+1)^2}{4(k-t-1)^{2c-5}}$. Since
	\begin{equation*}
		\begin{aligned}	
			\varphi(c,k,t)\ge1-\dfrac{1}{(t+2)^3}-\dfrac{3(t+1)}{(t+2)^3}-\dfrac{9(t+1)^2}{4(t+2)^7}>0
		\end{aligned}
	\end{equation*}
	for $k\ge2t+3$, and
	\begin{equation*}
		\begin{aligned}	
		\varphi(c,k,t)\ge1-\dfrac{1}{16t^4}-\dfrac{3}{8t^3}-\dfrac{9}{512t^6}>0
		\end{aligned}
	\end{equation*}
	for $k\le2t+2$,
	we have $(f_0(c,k,t))^2-h_3(c,k,t)>0$, as desired.

	(\ro4) Note that
	\begin{equation*}
		\begin{aligned}
			\dfrac{(f_2(c,k,t))^2-h_4(c,k,t)}{(\tt)^2}\ge&1-\dfrac{2(t+1)(t+2)}{(k-t-1)^{c-1}}-\dfrac{12(t+1)(k-t-1)^3}{{(k-t-1)c\c c}}-\dfrac{9(t+1)^2}{4(k-t-1)^{2c-6}}.
		\end{aligned}
	\end{equation*}
	 Suppose $k\ge2t+3$. Then	\begin{equation*}
		\begin{aligned}
			\dfrac{(f_2(c,k,t))^2-h_4(c,k,t)}{(\tt)^2}\ge&1-\dfrac{2(t+1)}{(t+2)^{4}}-\dfrac{12(t+1)}{(t+2)^{3}}-\dfrac{9(t+1)^2}{4(t+2)^{6}}\ge\dfrac{4}{81}>0.
		\end{aligned}
	\end{equation*}
	
	Now suppose $k\le2t+2$. 	For an integer $x\ge2$, we have
	$$\dfrac{x^3}{{xc\c c}}\cdot\dfrac{{(x+1)c\c c}}{(x+1)^3}=\left(\dfrac{x}{x+1}\right)^3\prod_{i=1}^c\dfrac{xc+i}{(x-1)c+i}\ge\left(\dfrac{x}{x+1}\right)^3\left(\dfrac{x+1}{x}\right)^c>1,$$
	implying that $\frac{x^3}{{xc\c c}}\le\frac{8}{{2c\c c}}$.
	If $t=1$, then
		\begin{equation*}
		\begin{aligned}
			\dfrac{(f_2(c,k,t))^2-h_4(c,k,t)}{(\tt)^2}>&1-\dfrac{3}{16}-\dfrac{192}{{2c\c c}}-\dfrac{9}{256}>0.
		\end{aligned}
	\end{equation*}
	If $t\ge2$, then
		\begin{equation*}
		\begin{aligned}
			\dfrac{(f_2(c,k,t))^2-h_4(c,k,t)}{(\tt)^2}>&1-\dfrac{3}{16t^2}-\dfrac{3}{2t^3}-\dfrac{9}{256t^6}>0.
		\end{aligned}
	\end{equation*}
	Then the desired result follows.
\end{proof}

\begin{lem}\label{f0f1}
	Let $c$, $k$ and $t$ be  positive integers with $c\ge6$ and $k\ge t+3$. If $c\ge4\log_2t+7$ or $k\ge2t+3$, then the following hold.
	\begin{itemize}
		\item[{\rm(\ro1)}] $f_1(c,k,t)>f_0(c,k,t)$.
			\item[{\rm(\ro2)}] $k\ge2t+4$ and $f_1(c,k,t)>f_2(c,k,t)$.
				\item[{\rm(\ro3)}] $k\le2t+3$ and $f_1(c,k,t)\le f_2(c,k,t)$. Equality holds if and only if $k=t+3$ or $(k,t)=(5,1)$.
	\end{itemize}
\end{lem}
\begin{proof}

	Suppose $T\in U^{[ck]}_{c,t}$ and $M\in U^{[ck]}_{c,k-1}$ with $T\subseteq M$. For $j\in\{t,t+1,\dots,k-1\}$, let
	$$\ml_j(T,M)=\left\{(I,F)\in U^{[ck]}_{c,j}\times\uu: T\subseteq I\subseteq M,\ I\subseteq F\right\},$$
	$$\ma_j(T,M)=\{F\in\uu:T\subseteq F,\ |M\cap F|=j\}.$$
	
	(\ro1)  For $j\in\{t,t+1,\dots,k-1\}$, we have
	\begin{equation*}
		{k-t-1\c j-t}\cdot\left(\dfrac{1}{(k-j)!}\prod_{i=j}^{k-1}{(k-i)c\c c}\right)=|\ml_j(T,M)|=\sum_{i=j}^{k-1}{i-t\c j-t}|\ma_i(T,M)|.
		\end{equation*}
	 Observe that
		\begin{equation}\label{dc}
			\begin{aligned}
			f_0(c,k,t)&=|\ml_{t+1}(T,M)|-|\ml_{t+2}(T,M)|=\sum_{i=1}^{k-t-1}\dfrac{3i-i^2}{2}|\ma_{t+i}(T,M)|\\
			&\le|\ma_{t+1}(T,M)|+|\ma_{t+2}(T,M)|.
			\end{aligned}
			\end{equation}
	This together with
	$$\ma_{t+1}(T,M)\sqcup\ma_{t+2}(T,M)\subseteq\{F\in\uu: T\subseteq F,\ |F\cap M|\ge t+1\}\subsetneq\mn_1(T,M,M)$$
	yields the desired result.

	(\ro2)
	Suppose $k\ge2t+4$. By \eqref{f0} and \eqref{f2}, we have
	\begin{equation*}
		\begin{aligned}
			\dfrac{f_0(c,k,t)-f_2(c,k,t)}{\tt}
			&=(k-2t-3)-\dfrac{(k-t-1)}{{(k-t-1)c\c c}}{k-t-1\c2}-\dfrac{(t+1)(k-t-1)}{{(k-t-1)c\c c}}\\
			&>1-\dfrac{1}{2(k-t-1)^{c-3}}-\dfrac{t+1}{(k-t-1)^{c-1}}\\
			&\ge1-\dfrac{1}{2(t+3)^3}-\dfrac{t+1}{(t+3)^5}>0.
		\end{aligned}
	\end{equation*}

	(\ro3) If $k=t+3$, then $\mn_1(T,M,M)=\mn_2(M)$, implying that $f_1(c,k,t)=f_2(c,k,t)$. In the following, assume that $k\ge t+4$. Observe that
	$$\left|\left\{F\in\uu: T\not\subseteq F,
	\ |F\cap M|=k-2\right\}\right|=t\left(\theta(c,k,k-2)-1\right).$$
	
	We first consider the case $k\le2t+2$. Then $t\ge2$.
	It is routine to check that
	\begin{equation}\label{f1u}
		\begin{aligned}
			f_1(c,k,t)\le\left(k-t-1\right)\tt+t(\theta(c,k,k-2)-1).
		\end{aligned}
	\end{equation}
	Then from $k\le2t+2$, \eqref{f2} and \eqref{f1u}, we have
	\begin{equation*}\label{f2-f1}
		\begin{aligned}
			\dfrac{f_2(c,k,t)-f_1(c,k,t)}{\ttt}
			&>\dfrac{1}{k-t-1}{(k-t-1)c\c c}-(t+1)-\dfrac{t\theta(c,k,k-2)}{\ttt}\\
			&>64t^4-2t-t>0,
		\end{aligned}
	\end{equation*}
	as desired.

In the following, assume that $k=2t+3$.
By \eqref{f0}, \eqref{f2} and \eqref{dc}, we have
\begin{equation*}
	\begin{aligned}
		f_1(c,k,t)-f_0(c,k,t)&=\sum_{i=3}^{t+2}{i-1\c2}|\ma_{t+i}(T,M)|+t(\theta(c,2t+3,2t+1)-1),\\
		f_2(c,k,t)-f_0(c,k,t)&={t+1\c2}\theta(c,2t+3,t+2).
\end{aligned}
\end{equation*}
Suppose $t=1$. We have $k=5$ and
\begin{equation*}
	\begin{aligned}
		f_1(c,k,t)-f_0(c,k,t)&=|\ma_4(T,M)|+(\theta(c,5,3)-1)=\theta(c,5,3)=f_2(c,k,t)-f_0(c,k,t).
	\end{aligned}
\end{equation*}
Then $f_1(c,k,t)=f_2(c,k,t)$. Next assume that $t\ge2$.

For $i\in\{3,4,\dots,t+2\}$, write
$$\lambda(i)={i-1\c2}{t+2\c i}\theta(c,2t+3,t+i).$$
When $3\le i\le t+1$, we have
\begin{equation*}
	\begin{aligned}
		\dfrac{\lambda(i+1)}{\lambda(i)}&=\dfrac{i}{i-2}\cdot\dfrac{t+2-i}{i+1}\cdot\dfrac{t+3-i}{{(t+3-i)c\c c}}<\dfrac{i}{(i-2)(i+1)}\cdot\dfrac{1}{(t+3-i)^{c-2}}<\dfrac{3}{4}.
	\end{aligned}
\end{equation*}
Therefore
\begin{equation*}
	\begin{aligned}
		f_1(c,k,t)-f_0(c,k,t)&<\lambda(3)\sum_{i=0}^\infty\left(\dfrac{3}{4}\right)^i+t(\theta(c,2t+3,2t+1)-1)\\
		&=4{t+2\c3}\theta(c,2t+3,t+3)+t(\theta(c,2t+3,2t+1)-1).
	\end{aligned}
\end{equation*}
We further get
\begin{equation*}
	\begin{aligned}
		\dfrac{f_2(c,k,t)-f_1(c,k,t)}{\theta(c,2t+3,t+3)}&>\dfrac{1}{t+2}{t+1\c2}{({t+1})c\c c}-4{t+2\c3}-\dfrac{t\theta(c,2t+3,2t+1)}{\theta(c,2t+3,t+3)}\\
		&>\dfrac{3}{4}(t+1)^6-2t^3-t>0.
	\end{aligned}
\end{equation*}
This finishes our proof.
\end{proof}

\medskip
\noindent{\bf Acknowledgment}	
\medskip

T. Yao is supported by Natural Science Foundation of Henan (252300420899). 
M. Cao is supported by the National Natural Science
Foundation of China (12301431).

\end{document}